\documentclass{article}

\usepackage[margin=1in]{geometry}
\usepackage{latexsym}
\usepackage{indentfirst}
\usepackage{graphicx}
\usepackage{amsmath,amsfonts}
\usepackage{amssymb}
\usepackage{amsthm}
\usepackage{bm}
\usepackage{accents}
\usepackage{changepage}
\usepackage{color}
\usepackage{bigints}
\usepackage{enumerate}
\usepackage{amsmath, nccmath}
\usepackage{yhmath}
\usepackage{dirtytalk}
\usepackage{hyperref}
\usepackage{cleveref}

\newtheorem{prop}{Proposition}[section]
\newtheorem*{defi*}{Definition}
\newtheorem{defi}[prop]{Definition}

\newtheorem{rem}[prop]{Remark}
\newtheorem{thm}[prop]{Theorem}
\newtheorem{coro}[prop]{Corollary}

\numberwithin{equation}{section}

\title{Construction of Hunt processes by the Lyapunov method and applications to generalized Mehler semigroups}

\author{Lucian Beznea\footnote{Simion Stoilow Institute of Mathematics  of the Romanian Academy,
 Research unit No. 2, 
P.O. Box \mbox{1-764,} RO-014700 Bucharest, Romania, and 
POLITEHNICA Bucharest, CAMPUS Institute
(e-mail: lucian.beznea@imar.ro)},
Iulian C\^{i}mpean\footnote{Department of Mathematics, Faculty of Mathematics and Computer Science, University of Bucharest, 14 Academiei, 010014~Bucharest, Romania and Simion Stoilow Institute of Mathematics of the Romanian Academy, 21~Calea Grivi\c{t}ei, 010702~Bucharest, Romania
(email: iulian.cimpean@unibuc.ro; iulian.cimpean@imar.ro)},
Michael R\"ockner\footnote{Fakult\"at f\"ur Mathematik, Universit\"at Bielefeld,
Postfach 100 131, D-33501 Bielefeld, Germany, and Academy for Mathematics and Systems Science, CAS, Beijing 
(e-mail: roeckner@mathematik.uni-bielefeld.de)}}

\date{}

\begin{document}

\maketitle

\paragraph{Abstract.}
It is known that in general, generalized Mehler semigroups defined on a Hilber space $H$ may not correspond to càdlàg (or even càd) Markov processes with values in $H$ endowed with the norm topology.
In this paper we deal with the problem of characterizing those generalized Mehler semigroups that do correspond to càdlàg Markov processes, which is highly non-trivial and has remained open for more than a decade. 
Our approach is to reconsider the {\it càdlàg problem} for generalized Mehler semigroups as a particular case of the much broader problem of constructing Hunt (hence càdlàg and quasi-left continuous) processes from a given Markov semigroup. Following this strategy, a consistent part of this work is devoted to prove that starting from a Markov semigroup on a general (possibly non-metrizable) state space, the existence of a suitable Lyapunov function with relatively compact sub/sup-sets in conjunction with  a local Feller-type regularity of the resolvent are sufficient to ensure the existence of an associated càdlàg Markov process; if in addition the topology is locally generated by potentials, then the process is in fact Hunt. Other results of fine potential theoretic nature are also pointed out, an important one being the fact that the Hunt property of a process is stable under the change of the topology, as long as it is locally generated by potentials. 
Based on such general existence results, we derive checkable sufficient conditions for a large class of generalized Mehler semigroups in order to posses an associated Hunt process with values in the original space, in contrast to previous results where an extension of the state space was required; to this end, we first construct explicit Lyapunov functions whose sub-level sets are relatively compact with respect to the (non-metrizable) weak topology, and then we use the above mentioned stability to deduce the Hunt property with respect to the stronger norm topology.
As a particular example, we test these conditions on a stochastic heat equation on $L^2(D)$ whose drift is given by the Dirichlet Laplacian on a bounded domain $D \subset \mathbb{R}^d$, driven by a (non-diagonal) Lévy noise whose characteristic exponent is not necessarily Sazonov continuous; in this case, we construct the corresponding Mehler semigroup and we show that it is the transition function of a Hunt process that lives on the original space $L^2(D)$ endowed with the norm topology.\\

\\

\noindent
{\bf Mathematics Subject Classification (2020):} 
60H15,     % Stochastic partial differential equations [See also 35R60]
%60H10,     % Stochastic ordinary differential equations [See also 34F05]
60J45,  	% Probabilistic potential theory
60J35,    	% Transition functions, generators and resolvents
60J40,  	% Right processes
% 60J57,  	% Multiplicative functionals
% 31C25,      % Dirichlet spaces
% 47D07,  	% Markov semigroups and applications to diffusion processes 
60J25,      % Continuous-time Markov processes on general state spaces
% 60J60.      % Diffusion processes
60J65,  % Brownian motion
%47D03.  	%Groups and semigroups of linear operators 
47D06,  % One-parameter semigroups and linear evolution equations
%47A35  	%Ergodic theory
%37A30   %Ergodic theorems, spectral theory, Markov operators
%37C40   %Smooth ergodic theory, invariant measures
%37L40   %Invariant measures (Infinite-dimensional dissipative dynamical systems)
%60J55  	Local time and additive functionals
%82B10  	Quantum equilibrium statistical mechanics (general)
%31C05  	Harmonic, subharmonic, superharmonic functions
% 60J60, Diffusion processes
%35R60,  % Partial differential equations with randomness, stochastic partial differential equations
31C40.  %Fine potential theory; fine properties of sets and functions
\section{Introduction}
The theory of constructing Hunt Markov processes on a general state space is nowadays well developed, offering diverse tools (be they purely probabilistic, based on Dirichlet forms, or of potential theoretic nature) that can be applied in various situations; see e.g. \cite{BlGe68}, \cite{Sh88} or \cite{BeBo04a} for potential theoretic methods, \cite{FuOsTa11} for the method of symmetric Dirichlet forms, \cite{MaRo92} for the non-symmetric case, 
or \cite{St99b} and \cite{St99a}, 
and more recently \cite{HaStTr22} for the non-sectorial case.
Recall that in addition to the c\`adl\`ag property of the trajectories, a Hunt process is strong Markov and quasi-left continuous (see \Cref{defi:Hunt} below), hence it has no predictable jumps. 
Furthermore, it is sufficiently regular to be further exploited with fine tools from martingale theory and stochastic calculus; one ultimate goal here is eventually to show that the constructed Markov process renders a (weak) solution to a certain stochastic differential equation. 
The construction of the process usually starts from some analytic object which is given in advance, namely the semigroup (or the transition function), the resolvent, the generator, or the Dirichlet form. 
In this paper, we are concerned with Hunt Markov processes constructed from a given transition semigroup or the corresponding resolvent of probability kernels. 
The existence of such a process for a given transition semigroup is far from being trivial, even when the semigroup is strong Feller, as discussed in detail in \cite{BeCiRo24}.
In a nutshell, the state-of-the-art methods of constructing Hunt Markov processes are based on three types of regularity conditions that the underlying space and the transition function (or the resolvent) should fulfill:
\begin{enumerate}
    \item[(i)] The topological state space is homeomorphic to a Borel subset of a compact metric space (i.e. it is a Lusin topological space), in particular it is metrizable.
    This condition arises technically from the (probabilistic) need that the underlying topology is generated by a {\it countable} family of continuous functions that separate the points.
    \item[(ii)] The transition operators (or the resolvent) have the Feller property (i.e. they map continuous bounded functions into continuous bounded functions), or some variant of it. 
    This kind of condition is one important ingredient to deduce the strong Markov property.
    At a more subtle level, it is crucial in order to derive that the constructed Markov process renders a right-continuous solution to the corresponding martingale problem, for a certain class of test functions like the image of bounded measurable functions under the resolvent operators, i.e. the domain of the weak Dynkin generator.
    \item[(iii)] There is a sequence of increasing family of compact sets that {\it exhaust} the state space in the sense of some associated capacity, called  {\it nest of compacts}. 
    This condition is known in the literature to be not only sufficient, but also necessary in order for the process to have c\`adl\`ag paths with values in the underlying topological state space; see e.g. \cite{LyRo92} and \cite{BeBo05}.
\end{enumerate}
Conditions (i)-(iii) from above can be rigorously formalized and checked in many situations, especially when the state space is finite dimensional. 
In infinite dimensions settings, e.g. when one aims to construct Markov solutions to certain SPDEs by starting from the transition function, the above conditions become significantly hard to check, in some cases apparently impossible.
Having in mind the above three typical sufficient conditions, the reason for this difficulty is quite simple: The bounded sets are not relatively compact in infinite dimensions, e.g. on a Hilbert space, so proving the existence of a nest of compacts is much more challenging than in finite dimensions. Bounded sets are nevertheless relatively weakly compact, but the weak topology is non-metrizable, and moreover, showing the Feller property with respect to the weak topology is in principle very hard, or even impossible.
And even if one is able to construct a Markov process which is Hunt (or merely c\`adl\`ag) with respect to the weak topology, going further and showing that the process is in fact Hunt (or merely c\`adl\`ag) with respect to the strong (norm) topology is not easy at all, simply because the (square of the) norm is not weakly continuous, and in most cases, neither in the domain of the corresponding infinitezimal generator nor in the image of the resolvent.

In this paper we aim to address systematically the above issues for constructing Hunt Markov processes in general (possibly non-metrizable) state spaces, by employing the method of Lyapunov functions; see \Cref{thm:main_construction} below. 
The second aim is to test our techniques on {\it generalized Mehler semigroups} which have been intensively studied e.g. by \cite{Jn87}, \cite{BoRoSc96}, \cite{FuRo00}, \cite{LeRo04}, \cite{BrZa10}, \cite{Kn11}, \cite{PrZa11}, \cite{LiZh12}, \cite{PeZa13}, \cite{Ap15} \cite{Ri15}, \cite{ShRo16}, \cite{Ba17}, and this is just a short list.
More precisely, we shall provide answers to the problem of existence of c\`adl\`ag Markov processes associated with  such semigroups, which has been open for some time (see e.g. \cite{Bretal10}, \cite[pg. 99]{PrZa11},  \cite[Question 4, pg. 723]{PeZa13}, or \cite[pg. 40]{Ap15}).
We shall derive conditions for general L\'evy noises under which the Markov process is not just c\'adl\'ag but also Hunt; see \Cref{thm:Mehler} and \Cref{coro:application} from \Cref{S:Mehler}.
Such processes are infinite dimensional generalizations of L\'evy-driven Ornstein-Uhlenbeck processes, and they serve as a role model of infinite dimensional processes.
To explain our method and the technical issues that we have to overcome in infinite dimensions, let us look first at the $d$-dimensional case where the situation is much easier to grasp.
Recall that a {\it L\'evy-driven Ornstein-Uhlenbeck process} is the solution to a stochastic differential equation (SDE) of the form
\begin{equation} \label{eq:L-OU-Rd}
    dX^x(t)=AX^x(t) dt + d Z(t), \quad t>0, X(0)=x\in \mathbb{R}^d,
\end{equation}
where $Z(t),t\geq 0$,  is a L\'evy process in $\mathbb{R}^d$ defined on a filtered probability space $(\Omega,\mathcal{F},\mathcal{F}_t,\mathbb{P})$ satisfying the usual hypotheses, whilst $A\in \mathbb{R}^{d\times d}$ is a real matrix; see \cite{Ap09} or \cite{Sa99} for details.
Once the L\'evy process is given, it is very easy to solve explicitly the SDE \eqref{eq:L-OU-Rd} for any $x\in \mathbb{R}^d$ by the stochastic generalization of the variation of constants formula, namely
\begin{equation} \label{eq:sol-LOU-Rd}
    X^x(t)=e^{tA}x+\int_0^t e^{(t-s)A} d Z(s), \quad t\geq 0,
\end{equation}
in particular the process $X^x$ is $a.s.$ c\`adl\`ag with values in $\mathbb{R}^d$, in fact $X$ can be realized as a (Hunt) Markov process; see \Cref{defi:Hunt}.
There is another way of constructing the solution $X$ to SDE \eqref{eq:L-OU-Rd}, which is far more abstract, but the advantage is that it is sufficiently general to be lifted to infinite dimensions, after some technical upgrade.
It is based on the folklore principle that instead of directly constructing a c\`adl\`ag process, it should be an easier task to construct first its finite dimensional distributions, and then to look for further conditions that guarantee path regularity for (a modification of) the corresponding process obtained e.g. by Kolmogorov's construction. 
In our case, this principle goes as follows (see e.g. \cite[Section 17]{Sa99}): Let $Z$ have the triplet $(b,R,M)$, so that the {\it characteristic exponent} of $Z$ is given by
\begin{align}
    &\lambda:\mathbb{R}^d\rightarrow \mathbb{C}\\
    &\lambda(a)=-i\langle a, b\rangle + \frac{1}{2} \langle a, Ra\rangle - \int_{\mathbb{R}^d} \left(e^{i\langle a, b\rangle}-1-\frac{i\langle a, z\rangle}{1+\|z\|^2}\right) M(dz),
\end{align}
where $b\in \mathbb{R}^d$ is the drift, $R\in \mathbb{R}^{d\times d}$ is the covariance matrix, whilst $M$ is the L\'evy measure.
Recall that $\lambda$ is related to $Z$ by the L\'evy-Kynthcine formula
\begin{equation}
    \mathbb{E}\left\{ e^{i\langle\xi,Z(t)\rangle}\right\}=e^{-t\lambda(\xi)}, \quad \xi \in \mathbb{R}^d, t\geq 0.
\end{equation}
Further, by denoting
\begin{equation}
    Y(t):=\int_0^t e^{(t-s)A} d Z(s), \quad \mu_t:=\mathbb{P}\circ Y(t)^{-1}, \quad 
    \hat{\mu}_t(\xi):=\mathbb{E}\left\{ e^{i\langle\xi,Y(t)\rangle}\right\}, \quad t\geq 0, \xi\in \mathbb{R}^d,
\end{equation}
one can easily see that $\mu_t$ is uniquely determined by $\lambda$ and $A$, since it is uniquely determined by its Fourier transform $\hat{\mu}_t(\cdot)$ and moreover
\begin{equation}\label{eq:mu_hat}
    \hat{\mu}_t(\xi)=e^{-\int_0^{t}\lambda(e^{sA^\ast}\xi) ds}, \quad t\geq 0,\xi \in \mathbb{R}^d,
\end{equation}
where $A^\ast$ is the transpose of $A$.
Furthermore, $\mu_t,t\geq 0$,  has the following properties:
\begin{enumerate}
    \item[(i)] $\mu_0=\delta_0$ and $\mu_{t+s}=\mu_t\ast \left[\mu_s\circ\left (e^{tA}\right)^{-1}\right], t,s\geq 0$, i.e. $(\mu_t)_{t\geq 0}$ is a $e^{tA}$-{\it convolution} semigroup;
    \item[(ii)] $[0,\infty)\ni t\mapsto\mu_t$ is weakly continuous as a path in the space of probability measures on $\mathbb{R}^d$.
\end{enumerate}
Furthermore, we can define the operators
\begin{equation}\label{eq:P_t}
P_tf(x):=\int_{\mathbb{R}^d}f\left(e^{tA}x+y\right) \mu_t(dy), \quad t\geq 0, x\in \mathbb{R}^d,
\end{equation}
acting on bounded and measurable functions $f:\mathbb{R}^d\rightarrow \mathbb{R}$.
It is easy to see that for every $f:\mathbb{R}^d\rightarrow \mathbb{R}$ which is bounded and continuous we have that
\begin{equation}
    [0,\infty)\times \mathbb{R}^d \ni(t,x)\mapsto P_tf(x)\in \mathbb{R} \quad \mbox{ is jointly continuous}.
\end{equation}
$(P_t)_{t\geq 0}$ defined in \eqref{eq:P_t} is called the {\it generalized Mehler semigroup} induced by the semigroup $(e^{tA})_{t\geq 0}$ and the family of probability measures $(\mu_t)_{t\geq 0}$ satisfying (i) and (ii) from above.
Clearly, in the $d$-dimensional case from above we already know that $(P_t)_{t\geq 0}$ admits a c\'adl\'ag Markov process since
\begin{equation}
    P_tf(x)=\mathbb{E}\left\{f\left(e^{tA}x+Y(t)\right)\right\}=\mathbb{E}\left\{f(X^x(t))\right\},
\end{equation}
where $X^x(t), t\geq 0$,  is given by \eqref{eq:sol-LOU-Rd}.
In the next lines we shall explain the other way around, namely the construction of $X$ starting from the given semigroup $(P_t)_{t\geq 0}$ (again, see e.g. \cite[Section 17]{Sa99}); this will also be the strategy in infinite dimensions, where constructing Hunt processes corresponding SDE in general situations is no longer easy.
The scheme is as follows: We start from the L\'evy triplet $(b,R,M)$ and the matrix operator $A$, and consider the generalized Mehler semigroup $(P_t)_{t\geq 0}$ associated with  $\left(e^{tA}\right)_{t\geq 0}$ and $(\mu_t)_{t\geq 0}$ given by its Fourier transform \eqref{eq:mu_hat}.
Then, we show that there exists a Hunt Markov process $X$ on $\mathbb{R}^d$ whose transition function is precisely $(P_t)_{t\geq 0}$.
We emphasize once again that the existence of such a process is highly nontrivial in general. 
However, for the $d$-dimensional case illustrated above, the construction of $X$ starting from $(P_t)_{t\geq 0}$ is immediate since one can easily check that $(P_t)_{t\geq 0}$ is in fact a $C_0$-semigroup on the space of continuous functions on $\mathbb{R^d}$ vanishing at infinity, so that the well-known result \cite[Theorem 9.4]{BlGe68} that ensures the existence of an associated Hunt process can be employed; see the last paragraph of \Cref{S:Hunt} for the detailed statement.
Then, using the regularity properties of $X$, one can show that $X$ solves the {\it martingale problem} associated with the infinitezimal generator induced by $(P_t)_{t\geq 0}$ described on a sufficiently large class of test functions.
Finally, the SDE \eqref{eq:L-OU-Rd} can be weakly solved by the well-known correspondence between the well-posedness of the martingale problem and the (weak) well-posedness of \eqref{eq:L-OU-Rd}.
Of course, one can also go even further and invoke the Yamada-Watanabe theory to derive the strong well-posedness of \eqref{eq:L-OU-Rd}.
The step that we shall be mainly concerned with in this paper is the construction of a Hunt (hence c\'adl\'ag) process $X$ from its semigroup $(P_t)_{t\geq 0}$, on general state spaces, in particular for generalized Mehler semigroups in infinite dimensions.
The easy problem of constructing (Hunt) L\'evy-driven OU processes in $\mathbb{R}^d$ changes dramatically in infinite dimensions, especially if the drift $A$ is unbounded and the L\'evy noise is cylindrical; it is well-known that this is a sensitive issue even in the Gaussian case.
To explain this, let us consider a fundamental example: Let $D\subset \mathbb{R}^d$ be an open bounded subset, $\Delta_0$ be the Dirichlet Laplacian defined on $D(\Delta_0) \subset L^2(D)$, and $(T_t)_{t\geq 0}$ the $C_0$-semigroup on $L^2(D)$ generated by $\Delta_0$.
Furthermore, let $W(t),t\geq 0$ be a (cylindrical) Wiener process on $L^2(D)$ with covariance operator $\Sigma$, in other words it is an infinite dimensional L\'evy process on $L^2(D)$ with L\'evy triplet $(0,\Sigma,0)$, where $I$ is the identity operator on $L^2(D)$. 
As it is well known, $W$ does not exist on $L^2(D)$, but merely on a larger space in which $L^2(D)$ is Hilbert-Schmidt embedded.
Nevertheless, if $(T_t)_{t\geq 0}$ is smoothing enough, then a stochastic generalization of the variation of constants formula still holds.
More precisely, by \cite[Theorem 5.11]{DaZa14}, if for each $T>0$ there exists $0<\alpha<1$ such that
\begin{equation}\label{eq:condition_stoch_conv}
   \int_0^T t^{-\alpha}\|T_t\Sigma\|_{\sf HS}^2 dt<\infty,
\end{equation}
where $\|\ \|_{\sf HS}$ denotes the {\it Hilbert-Schmidt norm}, then the {\it stochastic convolution}
\begin{equation}\label{eq:stoch_conv}
    \int_0^tT_{t-s} d W(s), t\geq 0,
\end{equation}
is well defined as a stochastic process which has a version with a.s. continuous paths in $L^2(D)$,
and
\begin{equation}
    X^x(t):=T_tx+\int_0^tT_{t-s} d W(s), t\geq 0,
\end{equation}
is the unique weak solution to the {\it stochastic heat equation}
\begin{equation}\label{eq:stoch_heat_eq}
    dX^x(t)=\Delta_0 X^x(t)dt + dW(t), \quad t>0, X^x(0)=x\in L^2(D).
\end{equation}
Using the fact that $\Delta_0$ admits an eigenbasis $(e_k,\lambda_k)_{k\geq 1}$ where $0>\lambda_1>\lambda_2>\dots$ are the eigenvalues of $\Delta_0$ satisfying the Weyl asympthotics:
\begin{equation}\label{eq:Weyl}
    \exists \ C>0: \quad -C^{-1} k^{2/d} \leq \lambda_k\leq -C k^{2/d}, \quad k\geq 1,
\end{equation}
one can easily check that condition \eqref{eq:condition_stoch_conv} is fulfilled if and only if $d=1$; see e.g. \cite[Example 5.7]{DaZa14}.
In fact, if we assume that $\|\Sigma e_k\|\lesssim k^\gamma$ for some $\gamma \in \mathbb{R}$, then a straightforward computation reveals that condition \eqref{eq:condition_stoch_conv} becomes $\gamma<1/d-1/2$ for all $d\geq 1$.
The construction and path-regularity of the stochastic convolution \eqref{eq:stoch_conv} are derived in the literature by stochastic calculus and essentially probabilistic analysis, by employing the {\it factorization method}; see \cite[Subsection 5.3]{DaZa14}, or \cite[Subsection 6.3]{LiRo15}. 
In this paper, we investigate infinite dimensional L\'evy-driven Ornstein-Uhlenbeck processes given by the linear SDE
\begin{equation}\label{eq:OU_eq_Levy}
    dX^x(t)=A X^x(t)dt + dZ(t), \quad t>0, X^x(0)=x\in H,
\end{equation}  
where $A$ is a (possibly unbounded) linear operator on a general Hilbert space $H$ which generates a $C_0$-semigroup, whilst $Z$ is a general (cylindrical) L\'evy noise on $H$.
Such SDEs with general L\'evy noise have been studied by the method of starting from the generalized Mehler semigroup e.g. in \cite{BoRoSc96}, \cite{FuRo00}, \cite{LeRo04} in the time-homogeneous case, or by \cite{Kn11} and \cite{ShRo16} for the time-inhomogeneous case.
Fundamentally, it was shown in \cite[Theorem 5.3]{FuRo00} that the generalized Mehler semigroup associated with the SDE \eqref{eq:OU_eq_Levy} can be represented by a c\`adl\`ag Markov process, but on a larger space $E$ in which $H$ can be Hilbert-Schmidt embedded.
When $A$ is bounded and $\lambda$ is {\it Sazonov continuous} (see \cite{Saz58}, \cite{Li82}, or \cite{FuRo00} for details on the Sazonov topology), then such a process can be constructed on $H$ by \cite[Corollary 5.5]{FuRo00}, however, this is not the case of \eqref{eq:OU_eq_Levy} in general.
For example, \cite[Theorem 5.3]{FuRo00} can not be applied in order to show that when $d=1$, the generalized Mehler semigroup associated with the stochastic heat equation \eqref{eq:stoch_heat_eq} can be represented by a Markov process that lives in $L^2(D)$, but merely on a larger space; however, we know by the classical factorization method mentioned above that such a process exists and has continuous paths in $L^2(D)$, hence \cite[Theorem 5.3]{FuRo00} offers space for improvements.

Having in mind the fact that the solution to the stochastic heat equation with cylindrical Wiener noise \eqref{eq:stoch_heat_eq} has continuous paths in $H=L^2(D)$ if $d=1$, one might suspect that the smoothing properties of $(T_t)_{t\geq 0}$ should have some impact on the path-regularity of the solution to \eqref{eq:OU_eq_Levy} given by \eqref{eq:sol-LOU-Rd}, for a cylindrical L\'evy noise $Z$ that may have non-zero jump part as well.
However, this is not true in general. 
More precisely, it was first shown in \cite[Theorem 2.1]{Bretal10} that if $Z(t)=\sum\limits_{n\geq 1} \beta_n Z^{(n)}(t)e_n, t\geq 0$, where $\{e_n\}_{n\geq 1}\subset D(A^\ast)$ is an orthonormal basis in a Hilbert space $H$, $A$ generates a $C_0$-semigroup on $H$, whilst $\left(Z^{(n)}\right)_{n\geq 1}$ are iid L\'evy processes on $\mathbb{R}$ with non-zero jump intensity, then the following negative result holds:  {\it if $(\beta_n)_{n\geq 1}$ does not converge to $0$, then with probability $1$, the trajectories of $X$ given by \eqref{eq:sol-LOU-Rd} have no point $t\in [0,\infty)$ at which the left or right limit exists in $H$}.
It was shown later on in \cite{PrZa11} that if $\left(Z^{(n)}\right)_{n\geq 1}$ are iid $\alpha$-stable L\'evy processes on $\mathbb{R}$ with $\alpha\in (0,2)$, then $\sum\limits_{n\geq 1} \beta_n^\alpha <\infty$ is a sufficient condition to ensure that $X$ (and in fact $Z$) has c\`adl\`ag paths in $H$.
In \cite{LiZh12} it was proved that the above condition is also necessary, and this characterization has been extended in \cite{LiZh16} for $\alpha$-semi-stable diagonal L\'evy noises. 
Thus, the smoothing property of $(T_t)_{t\geq 0}$ may have limited impact on the $H$-c\`adl\`ag regularity of the stochastic convolution with respect to a L\'evy noise in general, and in the particular case of $\alpha$-stable diagonal noises it has no impact at all. 
Further steps have been achieved in \cite[Theorem 3.1]{PeZa13}, showing that if the L\'evy noise takes values in the domain of some convenient (negative) power of $-A$ and some further moment bounds for the L\'evy measure are satisfied, then $X$ has a c\`adl\`ag modification in $H$ even if the noise $Z$ lives merely on a larger space; however, this result does not cover the diagonal $\alpha$-stable case from \cite{LiZh12}, as explained in \cite[Remark 3.5]{PeZa13}.
It is worth to mention that these results have been obtained by stochastic analysis tools like maximal inequalities for the norm of the stochastic convolution, whilst our approach relies on the construction of some convenient Lyapunov function which plays the role of the norm function, but it is fundamentally more flexible, in particular it is allowed to take infinite values on a set which is polar.

{\it To summarize, in general situations, the question if, and under which conditions, a Markov process associated to \eqref{eq:OU_eq_Levy} can be constructed having c\`adl\`ag paths in the original space $H$, has remained fundamentally open. In this paper we aim to address it from a potential theoretic perspective, starting from the generalized Mehler semigroup.}

The main issues in constructing Hunt Markov processes associated with  the generalized Mehler semigroups corresponding to \eqref{eq:OU_eq_Levy} emerge from the
difficulty of checking conditions (i)-(iii) from the beginning of this section.
Thus, the first aim of this paper is to revisit the general theory of constructing regular Markov processes in order to replace conditions (i)-(iii) with others that can be efficiently used in the case of non-metrizable topologies, in particular for generalized Mehler semigroups and weak topologies in Hilbert spaces.
The main general result in this direction is \Cref{thm:main_construction} from \Cref{S:Hunt}.
Based on this general result, we prove that under natural assumptions (more precisely under condition $\mathbf{(H_{A,\mu})}$ in \Cref{S:Mehler}), the generalized Mehler semigroup associated with \eqref{eq:OU_eq_Levy} is the transition function of a Hunt Markov process that lives on the original Hilbert space $H$ endowed with the norm topology; we do this in \Cref{thm:Mehler}.
One crucial step that had to be achieved is the construction of a proper Lyapunov function whose sublevel sets are compact with respect to the weak topology, and that the norm and the weak topologies are locally topologies that are generated by potentials; see \Cref{prop:tau_||}, \Cref{prop:V_H_mu_1}, and \Cref{prop:tau_w:quasi-U}.
As an application of \Cref{thm:Mehler}, more precisely in \Cref{coro:application}, we give explicit conditions under which the generalized Mehler semigroup associated with   
the SPDE
\begin{equation}\label{eq:stoch_heat_eq_Levy}
    dX^x(t)=\Delta_0 X^x(t)dt + dZ(t), \quad t>0, X^x(0)=x\in L^2(D), \quad D\subset \mathbb{R}^d,
\end{equation}
is the transition function of a Hunt Markov process on the original space $L^2(D)$ with respect to the norm topology; in this example, $Z(t),t\geq 0$, is a (non-diagonal and cylindrical) L\'evy noise on $L^2(D)$ whose characteristic exponent $\lambda:L^2(D)\rightarrow \mathbb{C}$ is given by
\begin{equation}
    \lambda(\xi):=\|\Sigma_1 \xi\|_{L^2(D)}^2+ \|\Sigma_2 \xi\|_{L^2(D)}^\alpha, \quad \xi \in L^2(D), \quad \alpha \in (0,2),
\end{equation}
whilst $\Sigma_1$ and $\Sigma_2$ are positive definite bounded linear operators on $L^2(D)$.
Some further fine regularity properties of the constructed Markov processes are obtained as byproducts of our potential theoretic approach.

The structure of the paper is the following: In \Cref{S:Resolvents} we provide a brief overview on the main probabilistic potential theoretic tools that we employ for proving the main results, we introduce the key notion of {\it quasi-$\mathcal{U}$ topology}, and discuss the stability of the Hunt property under changing the underlying topology.
In \Cref{S:Hunt} we state and prove the main result on constructing Hunt processes in general (possibly non-metrizable) topological spaces, by the method of Lyapunov functions.
Then, in \Cref{S:Mehler}, based on the general result obtained in \Cref{S:Hunt}, we derive conditions for a large class of generalized Mehler semigroups that ensure the existence of corresponding Hunt processes living on the original Hilbert space endowed with the norm topology. 
Finally, we test these conditions on a particular case, namely the one given by \eqref{eq:stoch_heat_eq_Levy} from above.

\section{Resolvents, right and Hunt processes}\label{S:Resolvents}
Throughout this section we follow mainly the terminology of \cite{BeBo04a}, but we also heavily refer to the classical works \cite{BlGe68}, \cite{Sh88}; see also the references therein. 

Let $(E, \mathcal{B})$ be a Lusin measurable space, i.e. it is measurably isomorphic to a Borel subset of a compact metric space endowed with the corresponding Borel $\sigma$-algebra.
We denote by $(b)p\mathcal{B}$ the set of all numerical, (bounded) positive $\mathcal{B}$-measurable functions on $E$.
Throughout, by $\mathcal{U}=(U_\alpha)_{\alpha>0}$ we denote a resolvent family of (sub-)Markovian kernels on $(E, \mathcal{B})$.
If $\beta>0$, we set $\mathcal{U}_\beta:=(U_{\alpha+\beta})_{\alpha>0}$.

\begin{defi} \label{defi 4.1}
An universally $\mathcal{B}$-measurable function $v:E\rightarrow \overline{\mathbb{R}}_+$ is called 
$\mathcal{U}$--{\rm excessive} 
%(w.r.t. $\mathcal{U}$) 
provided that  
$\alpha U_\alpha v \leq v$ for all $\alpha >0$ and $\mathop{\sup}\limits_\alpha \alpha U_\alpha v =v $ point-wise; by $\mathcal{E(\mathcal{U})}$ we denote the convex cone of all $\mathcal{B}$-measurable 
 $\mathcal{U}$--excessive functions.
 %w.r.t. $\mathcal{U}$.
\end{defi}

\begin{rem} \label{rem:U_exc}
Recall that if $f\in p\mathcal{B}$, then $U_\alpha f$ is $\mathcal{U}_\alpha$-excessive.
\end{rem}

If a universally $\mathcal{B}$-measurable function $w: E \rightarrow \overline{\mathbb{R}}_+$ is merely $\mathcal{U}_\beta$-supermedian (i.e. $\alpha U_{\alpha+\beta} w \leq w$ for all $\alpha > 0$), then its $\mathcal{U}_\beta$--{\it excessive regularization} 
%$\widehat{w} \in \mathcal{E(U)}$ 
is defined as
$\widehat{w}:=\mathop{\sup}\limits_\alpha \alpha U_{\beta +\alpha}w$. 
The function $\widehat{w}$ is 
$\mathcal{U}_\beta$--excessive and 
$\widehat{w} \in \mathcal{E(U_\beta)}$ provided that 
$w$ is  in addition $\mathcal{B}$-measurable.

\begin{defi} \label{defi: fine_topology}
The {\it fine topology} $\tau^f$ on $E$ (associated with $\mathcal{U}$) is the coarsest topology on $E$ 
such that every $\mathcal{U}_q$-excessive function is continuous for some (hence all) $q>0$.
\end{defi}

\begin{prop}\label{prop:zero_open}
    If $u$ is $\mathcal{U}_\beta$-excessive for some $\beta>0$, then the set $[u=0]$ is finely open.
\end{prop}

\begin{defi} \label{defi:natual_topology}
A topology $\tau$ on $E$ is called natural if it is a Lusin topology (i.e. $(E,\tau)$ is homeomorphic to a Borel subset of a compact metrizable space) which is coarser than the fine topology, and whose Borel $\sigma$-algebra is $\mathcal{B}$.    
\end{defi}

\begin{rem} \label{rem 4.3}
\begin{enumerate}[(i)]
    \item The necessity of considering natural topologies comes from the fact that, in general, the fine topology is neither metrizable, nor countably generated; see also \Cref{prop:quasi_left} below.
    \item Let $(E,\mathcal{B})$ be a Lusin measurable space, and $(f_n)_{n\geq 1}\subset \mathcal{B}$ be such that it separates the points of $E$. 
    Then the topology on $E$ generated by $(f_n)_{n\geq 1}$ is a Lusin topology; see e.g. \cite[Theorem A2.12]{Sh88} for the similar Radonian case.
\end{enumerate}
\end{rem}

There is a convenient class of natural topologies to work with (as we do in Section 2), especially when the aim is to construct a right process associated with  $\mathcal{U}$ (see Definition \ref{defi 4.4}). These topologies are called Ray topologies, and are defined as follows.

\begin{defi} \label{defi 4.5}
\begin{enumerate}[(i)]
\item If $\beta >0$ then a Ray cone associated with $\mathcal{U}_\beta$ is a cone $\mathcal{R}$ of bounded $\mathcal{U}_\beta$-excessive functions which is separable in the supremum norm, min-stable, contains the constant function $1$, generates $\mathcal{B}$, $U_\alpha(\mathcal{R}) \subset \mathcal{R}$ for all $\alpha > 0$, and $U_\beta((\mathcal{R}-\mathcal{R})_+) \subset \mathcal{R}$.
\item A {\it Ray topology} on $E$ is a topology generated by a Ray cone $\mathcal{R}$, and it is denoted by $\tau_{\mathcal{R}}$.
\end{enumerate} 
\end{defi}

In order to ensure the existence of a Ray topology, as well as many other useful properties related to the fine topology and excessive functions, we need the following condition:\\

\noindent
$\mathbf{(H_{\mathcal{C}})}.$ $\mathcal{C}$ is a min-stable convex cone of non-negative $\mathcal{B}(E)$-measurable functions such that
\begin{enumerate}[(i)]
    \item $1\in \mathcal{C}$ and there exists a countable subset of $\mathcal{C}$ which separates the points of $E$,
    \item $U_\alpha(\mathcal{C})\subset \mathcal{C}$ for all $f\in \mathcal{C}$ and $\alpha>0$,
    \item $\lim\limits_{\alpha\to\infty}\alpha U_{\alpha}f=f$ pointwise on $E$ for all $f\in \mathcal{C}$.
\end{enumerate}

The following result is basically Corollary 2.3 from \cite{BeRo11a}:
\begin{prop} \label{prop 2.5}
The following assertions are equivalent.
\begin{enumerate}[(i)]
    \item There exists a cone $\mathcal{C}$ such that $\mathbf{(H_{\mathcal{C}})}$ is fulfilled.
    \item $\mathcal{E}(\mathcal{U}_\beta)$ is min stable, contains the constant functions, and generates $\mathcal{B}$ for one (hence all) $\beta >0$.
    \item For any $\beta>0$ there exists a Ray cone associated with $\mathcal{U}_\beta$. 
\end{enumerate}
\end{prop}

For the rest of this section we assume that one (hence all) of the assertions from \ref{prop 2.5} is valid.
This is always satisfied when there is a right Markov process with resolvent $\mathcal{U}$, whose definition is given below.

\begin{rem} \label{rem: ray_topology}
\begin{enumerate}[(i)]
\item It is clear that any Ray topology is a natural topology.
A useful converse is true: For any natural topology there exists a finer Ray topology; see Proposition 2.1 from \cite{BeBo05}.
In fact, a key ingredient here is that given a countable family of finite $\mathcal{U}_\beta$-excessive functions, one can always find a Ray cone that contains it; we shall use this fact later on.
\item As a matter of fact, if $(\mathbf{H_\mathcal{C}})$ holds, then one can explicitly construct plenty of Ray topologies, following e.g. \cite{BeBo04a}, Proposition 1.5.1, or \cite{BeRo11a}, Proposition 2.2.
As we shall need such a construction later on, let us detail it here.
Let $\mathcal{A}_0\subset p\mathcal{B}_{b}$ be countable and such that it separates the set of finite measures on $(E,\mathcal{B})$.
The {\rm Ray cone generated by $\mathcal{A}_0$} and associated with  $\mathcal{U}_\beta$ is defined inductively as follows: 
\begin{align*}
&\mathcal{R}_0:=U_\beta (\mathcal{A}_0)\cup \mathbb{Q}_+,\\
&\mathcal{R}_{n+1}:=\mathbb{Q}_+\cdot \mathcal{R}_n \cup \left(\mathop{\sum}\limits_f\mathcal{R}_n \right)\cup \left(\mathop{\bigwedge}\limits_f \mathcal{R}_n\right) \cup \left(\mathop{\bigcup}\limits_{\alpha \in \mathbb{Q}_+}U_\alpha(\mathcal{R}_n)\right)\cup U_\beta ((\mathcal{R}_n-\mathcal{R}_n)_+),
\end{align*}
where by $\mathop{\sum}\limits_f \mathcal{R}_n$ resp. $\mathop{\bigwedge}\limits_f\mathcal{R}_n$ we denote the space of all finite sums (resp. infima) of elements from $\mathcal{R}_n$.
Then, the Ray cone $\mathcal{R}_{\mathcal{A}_0}$ generated by $\mathcal{A}_0$ is obtained by taking the closure of $\bigcup\limits_n \mathcal{R}_n$ w.r.t. the supremum norm.
\end{enumerate}
\end{rem}

\noindent
{\bf Right processes.} 
Let now $X=(\Omega, \mathcal{F}, \mathcal{F}_t , X(t), \theta(t) , \mathbb{P}^x)$ be a normal Markov process with state space $E$, shift operators $\theta(t):\Omega\rightarrow \Omega, \; t\geq 0$, and lifetime $\zeta$; if $\Delta$ denotes the cemetery point attached to $E$, then any numerical function $f$ on $E$ shall be extended to $\Delta$ by setting $f(\Delta)=0$. 

 We assume that $X$ has the resolvent $\mathcal{U}$ fixed above, i.e. for all $f\in b\mathcal{B}$ and $\alpha >0$
$$
U_\alpha f(x)=\mathbb{E}^{x} \! \int_0^{\infty}e^{-\alpha t}f(X(t))dt,\quad  x\in E.
$$

To each probability measure $\mu$ on $(E, \mathcal{B})$ we associate the probability 
$$\mathbb{P}^\mu (A):=\mathop{\int} \mathbb{P}^x(A)\; \mu(dx)$$
for all $A \in \mathcal{F}$, and we consider the following enlarged filtration
$$
\widetilde{\mathcal{F}}_t:= \bigcap\limits_\mu \mathcal{F}_t^\mu, \; \; \widetilde{\mathcal{F}}:= \bigcap\limits_\mu \mathcal{F}^\mu,
$$
where $\mathcal{F}^\mu$ is the completion of $\mathcal{F}$ under $\mathbb{P}^\mu$, and $\mathcal{F}_t^\mu$ is the completion of $\mathcal{F}_t$ in $\mathcal{F}^\mu$ w.r.t. $\mathbb{P}^\mu$; in particular, $(x,A)\mapsto\mathbb{P}^{x}(A)$ is assumed to be a kernel from $(E,\mathcal{B}^{u})$ to $(\Omega, \mathcal{F})$, where $\mathcal{B}^{u}$ denotes the $\sigma$-algebra of all universally measurable subsets of $E$.

\begin{defi} \label{defi 4.4}
The Markov process $X$ is called a right (Markov) process if the following additional hypotheses are satisfied:
\begin{enumerate}[(i)]
\item The filtration $(\mathcal{F}_t)_{t\geq 0}$ is right continuous and $\mathcal{F}_t=\widetilde{\mathcal{F}}_t, t\geq 0$.
\item For one (hence all) $\alpha>0$ and for each $f \in \mathcal{E}(\mathcal{U}_\alpha)$ the process $f(X)$ has right continuous paths $\mathbb{P}^{x}$-a.s. for all $x\in E$.
\item There exists a natural topology on $E$ with respect to which the paths of $X$ are $\mathbb{P}^{x}$-a.s. right continuous for all $x\in E$.
\end{enumerate}
\end{defi}

\begin{rem} 
 Usually in the literature, the space $E$ is endowed with a certain given topology, and the right continuity of the paths of the right process is considered in that topology. 
 In contrast, Definition \ref{defi 4.4} uses no a priori topology on $E$, which is endowed merely with the $\sigma$-algebra $\mathcal{B}$. 
It is the resolvent $\mathcal{U}$ which provides the topology, namely the fine topology (and its coarser natural or Ray topologies). 
However, there is no conflict between the two perspectives, and in fact, according to Theorem \ref{thm 4.6} from below, if the process has right-continuous paths w.r.t. some Lusin topology $\tau$ whose Borel $\sigma$-algebra is $\mathcal{B}$, then $\tau$ is necessary a natural topology.
Even more remarkable is the fact that a right process defined as above is stable under the change of the natural topology invoked in iii): if the process has right-continuous paths in some natural topology, then it has the same property in any natural topology.
It turns out that the regularity of the paths are stable at an even deeper level; we collect these results in \Cref{thm 4.6}, \Cref{coro:right_quasi_ray}, and \Cref{prop:quasi_left} below.
\end{rem}

The probabilistic description of the fine topology is given by the following key result, according to \cite{BlGe68}, Chapter II, Theorem 4.8, or \cite{Sh88}, Proposition 10.8 and Exercise 10.18, and \cite[Corollary A.10]{BeCiRo20}. 
The second assertion states  the relation between excessive functions and right-continuous supermartingales (cf. e.g.,  Proposition 1 from 
\cite{BeCi18a}, see also \cite{BeCi18}). 

\begin{thm} \label{thm 4.6} Let  $X$ be a right process and $f$ a universally $\mathcal{B}$-measurable function. 
Then the following assertions hold.

$(i)$ The function $f$ is finely continuous if and only if $(f(X(t)))_{t\geq 0}$ has $\mathbb{P}^{x}$-a.s. right continuous paths for all $x\in E$.
In particular, $X$ has a.s. right continuous paths in any natural topology on $E$.

$(ii)$ If $\beta >0$ then the  function $f$ is 
$\mathcal{U}_\beta$-excessive if and only if $(e^{-\beta t} f(X(t)))_{t\geq 0}$ is 
a right continuous 
$\mathcal{F}_t$-supermartingale w.r.t. $\mathbb{P}^{x}$ for all  $x\in E$.
\end{thm}

We end this paragraph by recalling the analytic and probabilistic descriptions of {\it polar} sets.
\begin{defi}
\begin{enumerate}[(i)]
\item If $u\in \mathcal{E}(\mathcal{U}_\alpha)$ and $A \in \mathcal{B}$, then the $\alpha$-order reduced function of $u$ on $A$ is given by
$$
R_\alpha^A u = \inf \{ v \in \mathcal{E}(\mathcal{U}_\alpha): \, v\geqslant u \mbox{ on } A \}.
$$
$R_\alpha^A u$ is merely 
$\mathcal{U}_\alpha$--supermedian  and we denote by $B_\alpha^A u :=\widehat{R_\alpha^A u}$ its 
$\mathcal{U}_\alpha$--excessive regularization, called the {\it balayage} of $u$ on $A$; 
however, if $A$ is finely open then $B_\alpha^A u=R_\alpha^A u\in\mathcal{E}(\mathcal{U_\alpha})$.

\item A set $A\in \mathcal{B}$ is called polar if $B_\alpha^A 1\equiv 0$.
\end{enumerate}
\end{defi}

\begin{rem} \label{rem:polar_characterization}
\begin{enumerate}
    \item[(i)] It is well-known that if $V\in \mathcal{E}(\mathcal{U}_\alpha)$ for some $\alpha\geq 0$ such that $U_1\left(1_{[V=\infty]}\right)\equiv 0$, then the set $[V=\infty]$ is polar.
    \item[(ii)] We have the following probabilistic characterization due to G.A. Hunt holds (see e.g. \cite{DeMe78}): If $X$ is a right process, then for all $u\in \mathcal{E}(\mathcal{U}_\alpha)$ and $A\in\mathcal{B}$
    \begin{equation}
        B_\alpha^A u=\mathbb{E}^x\{e^{-\alpha T_A} u (X(T_A))\},   
    \end{equation}
    where $T_A:=\inf \{ t>0 :  X(t)\in A \}$.
    In particular, $A$ is polar if and only if $\mathbb{P}^{x}(T_A<\infty)=0$ for all $x\in E$.
\end{enumerate}
\end{rem}

\paragraph{Existence of a right process on a larger space.}

Condition $\mathbf{(H_{\mathcal{C}})}$, although necessary, is not sufficient to guarantee the existence of a right process with the given resolvent $\mathcal{U}$, even if the resolvent is strong Feller; see \cite{BeCiRo24}. 
However, under $\mathbf{(H_{\mathcal{C}})}$, there is always a larger space on which an associated right process exists.

We denote by $Exc(\mathcal{U}_\beta)$ the set of all $\mathcal{U}_\beta$-excessive measures:
$\xi \in Exc(\mathcal{U}_\beta)$ if and only if $\xi$ is a $\sigma$-finite measure on $E$ and $\xi \circ \alpha U_{\alpha+\beta} \leq \xi$ for all $\alpha >0$.

\begin{defi}[The saturation of $E$ w.r.t. $\mathcal{U}_\beta$] \label{defi:saturation}
Let $\beta >0$.
\begin{enumerate}[(i)]
\item The {\it energy functional} associated with $\mathcal{U}_\beta$ is $L^{\beta}: Exc(\mathcal{U}_\beta)\times \mathcal{E}(\mathcal{U}_\beta) \rightarrow \overline{\mathbb{R}}_+$ given by
\begin{align}
L^{\beta}(\xi,v):&=\sup\{\mu(v) \; : \; \mu \mbox{ is a } \sigma\mbox{- finite measure, } \mu \circ U_\beta \leq \xi\}\\
&=\sup\{\xi(f) \; : \; f\geq 0 \mbox{ is a } \mathcal{B}\mbox{-measurable function such that } U_\beta f \leq v\}.
\end{align}
\item The {\it saturation} of $E$ (with respect to $\mathcal{U}_\beta$) is the set $E_1$ of all extreme points of the set $\{\xi \in Exc(\mathcal{U}_\beta)\; : \; L^{\beta}(\xi,1)=1\}$.
\item The map $E\ni x \mapsto \delta_x \circ U_\beta \in Exc(\mathcal{U}_\beta)$ is a measurable embedding of $E$ into $E_1$ and 
every $\mathcal{U}_\beta$-excessive function $v$ has an extension $\widetilde{v}$ to $E_1$, defined as $\widetilde{v}(\xi):=L^{\beta}(\xi,v)$.
The set $E_1$ is endowed with the $\sigma$-algebra $\mathcal{B}_1$ generated by the family $\{\widetilde{v}: \; v\in \mathcal{E}(\mathcal{U}_\beta)\}$.
In addition, as in \cite[Sections 1.1 and 1.2]{BeBoRo06}, there exists a unique resolvent of kernels $\mathcal{U}^{1}=(U^{1}_\alpha)_{\alpha>0}$ on $(E_1, \mathcal{B}_1)$ 
which is an {\it extension} of $\mathcal{U}$ in the sense that
\begin{enumerate}
    \item[(iii.1)] $U^1{1_{E_1\setminus E}}=0$ on $E_1$,
    \item[(iii.2)] $(U^1f)|_E=U(f|_E)$ for all $f \in b\mathcal{B}_1$.
\end{enumerate}
More precisely, it is given by
\begin{align} \label{eq 4.1}
U^{1}_\alpha f(\xi)&=L^{\beta}(\xi, U_\alpha (f|_E))\\
&=L^{\beta}(\xi,U_\beta(f|_E+(\beta-\alpha)U_\alpha))=\xi(f+(\beta-\alpha)U_\alpha f) 
\mbox{ for all } f \in bp\mathcal{B}_1, \xi\in E_1, \alpha >0.
\end{align}
\end{enumerate}
\end{defi}

\begin{rem}\label{rem:saturation}
Note that $(E_1,\mathcal{B}_1)$ is a Lusin measurable space, the map $x \mapsto\delta_x \circ U_\beta$ identifies $E$ with a subset of $E_1$, $E\in \mathcal{B}_1$,  and $\mathcal{B}=\mathcal{B}_1|_E$. 
Furthermore, $E$ is dense in $E_1$ with respect to the fine topology on $E_1$ associated with $\mathcal{U}^1$. 
Hence, if $\alpha>0$ and $v$ is a $\mathcal{U}_\alpha$-excessive function then its $\mathcal{U}^1_\alpha$-excessive extension $\tilde{v}$ given by \Cref{defi:saturation} is the unique extension of $v$ from $E$ to $E_1$ by fine continuity.
\end{rem}

The existence of a right process on a larger space, more precisely on $E_1$, is given by the following result, for which we refer to \cite[(2.3)]{BeRo11a}, in \cite[Sections 1.7 and 1.8]{BeBo04a},   \cite[Theorem 1.3]{BeBoRo06}, and \cite[Section 3]{BeBoRo06a}; see also
\cite{St89}.

\begin{thm} \label{thm 4.15}
There is always a right process $X^1$ on the saturation $(E_1,\mathcal{B}_1)$, associated with $\mathcal{U}^{1}$. Moreover, the following assertions are equivalent:
\begin{enumerate}[(i)]
\item There exists a right process on $E$ associated with $\mathcal{U}$.
\item The set $E_1\setminus E$ is polar (w.r.t. $\mathcal{U}^1$).
\end{enumerate}
\end{thm}

\paragraph{Natural extensions and saturated processes.} 
This paragraph is aimed to offer some qualitative insights on the saturation $E_1$ introduced above, from a slightly different perspective:
Given a right process $X$ on $E$ with resolvent $\mathcal{U}$, an interesting question that can be raised is whether this process can be extended to a larger space, in a natural way.
Such an extension can be defined as follows:

\begin{defi} \label{defi:saturated}
Let $(X, \mathbb{P}^x, x\in E)$ be a right process on $E$.
\begin{enumerate}[(i)]
    \item We say that a second right process $(\widetilde{X}, \widetilde{\mathbb{P}}^x, x\in \widetilde{E)}$ on $\widetilde{E}$ is a natural extension of $X$ if $E\in \mathcal{B}(\widetilde{E})$ and
\begin{enumerate}
    \item[(i.1)] $\widetilde{\mathbb{P}}^x(\widetilde{X}_t\in \widetilde{E}\setminus E)=0$ for all $t> 0$, $x\in \widetilde{E}$,
    \item[(i.2)] If $x\in E$, then $\widetilde{\mathbb{P}}^x \circ \widetilde{X}_t^{-1}=\mathbb{P}^x \circ X_t^{-1}$ for all $t\geq 0$.
\end{enumerate}
\item We say that $X$ is saturated if there is no proper natural extension of it.
\end{enumerate}
\end{defi}

The following result can be easily deduced from e.g. \cite{BeCiRo20}:

\begin{thm}\label{thm:saturation}
Let $X$ be a right process on $E$.
The following assertions hold:
\begin{enumerate}[(i)]
    \item The right process $X^1$ on the saturation $E^1$ given by Theorem \ref{thm 4.15} is maximal, in the sense that $X^1$ is a natural extension of any natural extension $\widetilde{X}$ of $X$.
    In particular, for any such natural extension $\widetilde{X}$ on $\widetilde{E}$, we have that $\widetilde{E}\setminus E$ is polar.
    \item The process $X$ is saturated if and only if $E=E^1$.
\end{enumerate}
\end{thm}

\begin{rem}
\Cref{thm:saturation} has two immediate consequences:
\begin{enumerate}
    \item[(i)] The saturation property depends only on the resolvent $\mathcal{U}$ (or the transition semigroup) and the underlying state space $E$. 
    In particular, if $X$ is a saturated right process, then so is any other right process sharing the same resolvent.
    \item[(ii)] One can construct many examples of right processes which are not saturated, simply as follows: If $X$ is a right process and $A$ is a polar set, then the restriction $X^A$ is not saturated. Perhaps the simplest example is the uniform motion to the right on $E=(0,\infty)$ which is not saturated, as one can easily see that its saturation is $E_1=[0,\infty)$. Another obvious example is obtained by taking $E:=\mathbb{R}^2\setminus\{0\}$ and $X$ to be the restriction to $E$ of the Brownian motion defined on $\mathbb{R}^2$; since $\{0\}$ is polar, we get that $X$ is not saturated on $E$. However, we shall see later on that the Brownian motion defined on $\mathbb{R}^2$ is saturated, as well as, e.g., any Hunt process on a locally compact metric space whose semigroup is Feller (in the strong sense), i.e. it is a $C_0$-semigroup on the space of continuous functions vanishing at infinity. 
    In fact, the saturation property can be deduced in much more general settings by the existence of a {\rm Lyapunov function vanishing at infinity}, see \Cref{thm:main_construction}, (iii), and \Cref{coro:lccb} below.
\end{enumerate}
\end{rem}

The saturation property of a right process has an interesting analytic regularity counterpart, which requires some context: Define the weak generator $({\sf L}, D_b({\sf L}))$ by
\begin{equation}\label{eq:D_b_L}
    D_b({\sf L}):=\{ U_\alpha f : f\in \mathcal{B}_b\}, \quad {\sf L}(U_\alpha f)= \alpha U_\alpha f -f, \quad f\in \mathcal{B}_b, \quad \mbox{for some } \alpha >0,
\end{equation}
and note that $({\sf L}, D_b({\sf L}))$ does not depend on $\alpha >0$.
It is easy to see that a $\sigma$-finite measure $\xi$ on $(E,\mathcal{B})$ is $\mathcal{U}_\beta$-excessive, i.e. $\xi\in Exc(\mathcal{U}_\beta)$ if and only if
\begin{equation}\label{eq:m_excessive}
    (L^\ast-\beta)\xi\leq 0 \mbox{ in the weak sense }, \mbox{ i.e. } \xi\left((L-\beta)f\right)\leq 0 \quad \mbox{ for all } f\in D_b({\sf L})\cap L^1(m).
\end{equation}

\begin{thm}[Analytic regularity of saturated processes]
Let $X$ be a right process on $E$ with resolvent $\mathcal{U}$. 
Then $X$ is saturated if and only if for any $\beta>0$ and any finite measure $\xi$ which satisfies \eqref{eq:m_excessive}, i.e. $\xi \in Exc(\mathcal{U}_\beta)$, there exists a $\sigma$-finite positive measure $\mu$ on $E$ such that $(L^\ast -\beta)\xi=-\mu$ in the weak sense, i.e. $\xi=\mu\circ U_\beta$.
\end{thm}
\begin{proof}
It is an immediate consequence of \Cref{thm:saturation} and \cite[Proposition 1.5.11]{BeBo04a}.
\end{proof}

\paragraph{Hunt processes.} 
In this paragraph we assume that $X$ is a right process on $E$ which is conservative, i.e. has $\mathbb{P}^x$-a.s. infinite lifetime, for all $x\in E$.

\begin{defi} \label{defi:Hunt}
Let $\tau$ be a topology on $E$ such that $\sigma(\tau)=\mathcal{B}$.
We say that $X$ is a quasi-left continuous with respect to $\tau$ if for every $\mu\in \mathcal{P}(E)$ and any sequence of $\mathcal{F}_t$-stopping times $(T_n)_{n\geq 1}$ such that $T_n \nearrow_n T$, we have that
\begin{equation} \label{eq:hunt}
 \tau-\lim_{n}X_{T_n} = X_T \mbox{ on } [T<\infty] \quad \mathbb{P}^\mu\mbox{-a.s.}
\end{equation}
If, in addition to quasi-left continuity, $X$ has also left $\tau$-limits in $E$ $\mathbb{P}^\mu$-a.s. for all $\mu\in \mathcal{P}(E)$, we say that $X$ is a Hunt process with respect to $\tau$.
\end{defi}

\begin{rem}
Suppose that $X$ is a Hunt process w.r.t. $\tau$, and let $d_\tau$ be a metric which generates $\tau$. 
Because $H:=\left[ \limsup\limits_n{d_{\tau}(X_{T_n},X_T)1_{[T<\infty]}}>0\right]\in \mathcal{F}$, we have that $\mathbb{P}^\cdot(H)$ is $\mathcal{B}^u$-measurable, hence \eqref{eq:hunt} holds $P^\mu$-a.s. for all probabilities $\mu$ on $E$. 
Moreover, it remains valid also for all $\mathcal{F}_t^\mu$-stopping times $T_n \nearrow_n T$; see e.g. \cite[Theorem 47.6]{Sh88},  or \cite[Theorem A.2.1]{FuOsTa11}.
\end{rem}

\begin{defi}\label{defi: U-nest}
An increasing sequence of $\mathcal{B}$-measurable sets $(F_n)_{n\geq 1}$ is called a $\mathcal{U}$-nest if there exists $\alpha>0$ and a strictly positive and bounded function $u\in \mathcal{E}_\alpha$ such that
\begin{equation} \label{eq:nest1}
    \lim\limits_n B_\alpha^{F_n^c}u =0 \quad \mbox{ on } E.
\end{equation}
\end{defi}

\begin{rem}\label{rem:prob_nest}
Recall that if $X$ is a right process with resolvent $\mathcal{U}$, then 
\begin{equation} \label{eq:nest2}
(F_n)_{n\geq 1} \mbox{ is a $\mathcal{U}$-nest if and only if } T_{F_n^c}\mathop{\nearrow}\limits_n\infty \quad \mathbb{P}^\mu \mbox{-a.s. for all } \mu\in \mathcal{P}(E).
\end{equation}    
\end{rem}

\begin{prop}
If $(F_n)_{n\geq 1}$ is a $\mathcal{U}$-nest and $E_0:=\bigcup\limits_n F_n$, then $E\setminus E_0$ is a polar set.
\end{prop}
\begin{proof}
Indeed, since $E\setminus E_0\subset F_n^c$, $n\geq 1$, we have that $B_1^{E\setminus E_0}1\leq B_1^{F_n^c}1\mathop{\searrow}\limits_n 0$.
\end{proof}

The following result resembles \cite[Proposition 4.1]{BeRo11a}; see also
\cite[Remark 3.3]{BeBoRo06a}.
The following assertions hold.
\begin{prop} \label{prop:nest}
Let $\alpha \geq 0$ and $V\in \mathcal{E}(\mathcal{U}_\alpha)$. % be $\mathcal{B}$-measurable. 
\begin{enumerate}[(i)]
    \item If $U_{1}\left(1_{[V=\infty]}\right)=0$, i.e. $[V=\infty]$ is polar according to \Cref{rem:polar_characterization}, then $F_n:=[V\leq n], n\geq 1$ is a $\mathcal{U}$-nest.
    \item If $V>0$ then $F_n:=[V\geq 1/n],$ 
    $ n\geq 1$, is a $\mathcal{U}$-nest.
\end{enumerate}
\end{prop}

\begin{proof}
The first statement follows by
\begin{equation*}
    B_\alpha^{F_n^c}1=B_\alpha^{F_n^c\cap [V<\infty]}1\leq B_\alpha^{F_n^c\cap [V<\infty]}(V/n)\leq 1_{[V<\infty]}V/n\xrightarrow[n\to \infty]{}0.
\end{equation*}
To prove (ii), note first that we can assume that $0<V<1$.
Then,
\begin{equation}
   B_\alpha^{F_n^c}V\leq B_\alpha^{F_n^c}(1/n)\leq 1/n \xrightarrow[n\to \infty]{}0. 
\end{equation}
\end{proof}

We introduce the following class of topologies which, in a few words, are generally weaker than natural or Ray topologies, may be non-metrizable, and are generated by functions of the type $U_\alpha f$ only locally on the elements of a $\mathcal{U}$-nest; as a candidate one can consider the weak toplogy on an infinite dimensional Hilbert space.
\begin{defi} \label{def:quasi-U}
\begin{enumerate}[(i)]
    \item We say that $\tau$ is a $\mathcal{U}$-topology on $E$ if there exist $\alpha_k\geq \alpha_1>0$ and $f_k\in \mathcal{B}(E)$, $k\geq 1$, and $V\in \mathcal{E}(\mathcal{U}_{\alpha_0})$ for some $\alpha_0<\alpha_1$ such that
    \begin{enumerate}
        \item[-] $\tau$ coincides with the topology on $E$ generated by the family $U_{\alpha_k} f_k, k\geq 1$
        \item[-] $U_{1}\left(1_{[V=\infty]}\right)=0$  and $|f_k|\leq V$ for all $k\geq 1$,
        \item[-] the family $(U_{\alpha_k} f_k)_k$ separates the set of measures $\{\mu\in\mathcal{P}(E) : \mu (U_{\alpha_k}|f_k|)<\infty\}$.
    \end{enumerate}
    \item We say that $\tau$ is a quasi-$\mathcal{U}$-topology on $E$ with attached $\mathcal{U}$-nest $(F_n)_{n\geq 1}$ if there exists a $\mathcal{U}$-topology $\overline{\tau}$ on $E$ such that
    \begin{equation*}
    \tau|_{F_n}=\overline{\tau}|_{F_n}, \mbox{ and } F_n \mbox{ is both } \tau \mbox{ and } \overline{\tau}\mbox{-closed} \mbox{ for all } n\geq 1.
    \end{equation*}
\end{enumerate}
Clearly, if $\tau$ is a $\mathcal{U}$-topology, then it is also a quasi-$\mathcal{U}$-topology
\end{defi}

\begin{rem} 
Trivially, if $(P_t)_t$ is a Feller semigroup on a locally compact topological space $(E,\tau)$ with countable base, then $\tau$ is a $\mathcal{U}$ topology.
\end{rem}

\begin{rem} \label{rem:quasi-metric}
\begin{enumerate}[(i)]
    \item If $\tau$ is a quasi-$\mathcal{U}$-topology, then its Borel $\sigma$-algebra on $E$ coincides with $\mathcal{B}$.
    \item By \Cref{rem 4.3}, any $\mathcal{U}$-topology is a natural topology on $E$ with respect to $\mathcal{U}$.
    A quasi-$\mathcal{U}$-topology on $E$ with attached $\mathcal{U}$-nest $(F_n)_{n\geq 1}$, $n\geq 1$, is not necessarily metrizable, hence it is not a natural topology on $E$ in general. 
\end{enumerate}
\end{rem}

\begin{rem} \label{rem: Ray_is_U}
The Ray topology induced by a Ray cone $\mathcal{R}_{\mathcal{A}_0}$ generated by a countable family of functions $\mathcal{A}_0\subset \mathcal{B}_p(E)$ which separates the points, in the sense of \Cref{rem: ray_topology}, (ii), is a $\mathcal{U}$-topology on $E$.
This can be easily seen because the topology induced by $\mathcal{R}_{\mathcal{A}_0}$ is in fact generated by a countable family of functions of the type $U_\alpha f$, for some $\alpha>0$ and $f\in p\mathcal{B}_b$. 
The converse is not true in general.
\end{rem}

An immediate result relating right processes and quasi-$\mathcal{U}$ topologies is the following.
\begin{coro}\label{coro:right_quasi_ray}
    If $X$ is a right process on $E$, then it has $\mathbb{P}^\mu$-a.s. right-continuous paths with respect to any natural topology, as well as with respect to any quasi-$\mathcal{U}$-topology, for all $\mu\in \mathcal{P}(E)$.
\end{coro}
\begin{proof}
Let $X$ be a right process and $\tau$ be a natural topology.
By \Cref{thm 4.6} we have that $X$ has $\mathbb{P}^x$-a.s. right continuous paths with respect to $\tau$ for all $x\in E$.
Now, by \cite[Ch. IV, Theorem 34]{DeMe78}, we have that
\begin{equation*}
    \Lambda_r:=\left\{\omega\in \Omega : [0,\infty)\ni t\mapsto X_t(\omega)\in E \mbox{ is not right continuous with respect to } \tau\right\} \in \mathcal{F}^u,
\end{equation*}
hence $E\ni x\mapsto \mathbb{P}^x(\Lambda_r)\in \mathbb{R}$ is $\mathcal{B}^u$-measurable, and we can rigorously infer that $\mathbb{P}^\mu(\Lambda_r)=\int_E \mathbb{P}^x(\Lambda_r) \; \mu(dx)=0$ for all $\mu\in \mathcal{P}(E)$.

Assume now that $\tau$ is a quasi-$\mathcal{U}$-topology with attached $\mathcal{U}$-nest $(F_n)_{n\geq 1}$ and $f_k,k\geq 1$ as in \Cref{def:quasi-U}.
Further, let $\overline{\tau}$ be the topology on $E$ generated by $U_{\alpha_k}f_k, k\geq 1$.
Since $\overline{\tau}$ is a natural topology on $E$ with respect to $\mathcal{U}$, by the first part we get that $X$ has $\mathbb{P}^\mu$-a.s. right-continuous paths with respect to $\overline{\tau}$ for all $\mu\in \mathcal{P}(E)$.
To conclude, note that by \Cref{rem:prob_nest} it is sufficient to show that for all $\mu\in \mathcal{P}(E)$ we have that $X$ has $\mathbb{P}^\mu$-a.s. right continuous paths with respect to $\tau$ on each time interval $\left[0,T_{E\setminus F_n}\right), n\geq 1$, which is now clear since on each $\left[0,T_{E\setminus F_n}\right)$ the process $X$ lies in $F_n$, and $\tau|_{F_n}=\overline{\tau}|_{F_n}$.
\end{proof}
 
The remarkable property of quasi-$\mathcal{U}$-topologies is that they preserve the existence of left limits, as well. More precisely, we have the following result:

\begin{prop} \label{prop:quasi_left} 
Let $X$ be a conservative right process on $E$. 
If $X$ has $\mathbb{P}^x$-a.s. left limits in $E$ with respect to one quasi-$\mathcal{U}$-topology on $E$ for all $x\in E$, then $X$ is a Hunt process with respect to any quasi-$\mathcal{U}$-topology on $E$.
\end{prop}

\begin{proof}
Assume that $X$ has $\mathbb{P}^x$-a.s. left limits in $E$ with respect to some quasi-$\mathcal{U}$-topology $\tau$ for all $x\in E$.

\medskip
{\bf Step I:} {\it $X$ has $\mathbb{P}^\mu$-a.s. c\`adl\`ag paths with respect to $\tau$ for all $\mu\in \mathcal{P}(E)$}.
Indeed, first of all, recall that by \Cref{coro:right_quasi_ray}, $X$ has $\mathbb{P}^\mu$-a.s. right-continuous paths with respect to any quasi-$\mathcal{U}$-topology on $E$ for all $\mu\in\mathcal{P}(E)$.

Further, let $\alpha_k, f_k,k\geq 1$ be as in \Cref{def:quasi-U}, and $\tilde{\tau}$ be the topology on $E$ generated by $U_{\alpha_k}f_k, k\geq 1$.
In particular, $\tilde{\tau}$ is a natural (hence Lusin) topology on $E$ with respect to $\mathcal{U}$. 
Also, by choosing $(F_n)_{n\geq 1}$ to be a  $\mathcal{U}$-nest attached to $\tau$ as in Definition \ref{def:quasi-U}, then by \Cref{rem:prob_nest} we have
\begin{equation*}
  T_{F_n^c}\mathop{\nearrow}\limits_n\infty \quad \mathbb{P}^x \mbox{-a.s. for all } x\in E.
\end{equation*}
Consequently, using that $F_n$ is $\tau$-closed for each $n\geq 1$, we deduce that $X$ has $\mathbb{P}^x$-a.s. c\`adl\`ag paths with respect to $\tilde{\tau}$, for all $x\in E$. 
Since $\tilde{\tau}$ is a Lusin topology, by \cite[Ch. IV, Theorem 34]{DeMe78} we have
\begin{equation*}
    \Lambda:=\left\{\omega\in \Omega : [0,\infty)\ni t\mapsto X_t(\omega)\in E \mbox{ is not c\`adl\`ag w.r.t. } \tilde{\tau} (\mbox{or }\tau)\right\} \in \mathcal{F}^u,
\end{equation*}
hence $E\ni x\mapsto \mathbb{P}^x(\Lambda)\in \mathbb{R}$ is $\mathcal{B}^u$-measurable, and we can rigorously deduce that  $\mathbb{P}^\mu(\Lambda)=\int_E \mathbb{P}^x(\Lambda) \mu(dx)=0$; Step I is therefore achieved.

\medskip
{\bf Step II:} {\it $X$ is quasi-left continuous with respect to $\tau$}. 
To prove this we adapt the argument from \cite[Lemma 3.21]{MaRo92}. 
Let $\mu \in \mathcal{P}(E)$ and $T_n\nearrow_n T$ be $(\mathcal{F}_t)$-stopping times; by writing $[T<\infty]=\cup_{n\geq 1} [T\leq n]$, we can and shall assume without loss of generality that $T$ is bounded. 
Furthermore, since $[V\leq n], n\geq 1$ forms a $\mathcal{U}$-nest according to \Cref{prop:nest}, by writing $[T<\infty]=\cup_{n\geq 1} [T\leq T_{V>n}]$ $\mathbb{P}^\mu$-a.s., we can and shall assume without loss of generality that $\sup\limits_{t\in [0,T]}V(X_t)$ and thus $\sup\limits_{t\in [0,T]}U_{\alpha_k}|f_k|(X_t)\leq \sup\limits_{t\in [0,T]}V(X_t)/(\alpha_k-\alpha_0)$ are bounded uniformly with respect to $k\geq 1$.
Then, using also Step I, we have for all $k,l\geq 1$:
\begin{align*}
\mathbb{E}^\mu\left[U_{\alpha_k}f_k(X_{T}) U_{\alpha_l}f_l(X_{T-})\right]
&=\lim\limits_n \mathbb{E}^\mu\left[U_{\alpha_k}f_k(X_{T}) U_{\alpha_l}f_l(X_{T_n})\right]\\
&=\lim\limits_n \mathbb{E}^\mu\left[\int_0^\infty e^{-\alpha_k t}f_k(X_{T+t}) dt \; U_{\alpha_l}f_l(X_{T_n})\right]\\
&=\lim\limits_n \mathbb{E}^\mu\left[e^{\alpha_k T}\int_T^\infty e^{-\alpha_k t}f_k(X_{t}) dt \; U_{\alpha_l}f_l(X_{T_n})\right]\\
&=\lim\limits_n \mathbb{E}^\mu\left[e^{\alpha_k T_n}\int_{T_n}^\infty e^{-\alpha_k t}f_k(X_{t}) dt \; U_{\alpha_l}f_l(X_{T_n})\right]\\
&=\lim\limits_n \mathbb{E}^\mu\left[U_{\alpha_k}f_k(X_{T_n}) U_{\alpha_l}f_l(X_{T_n})\right]\\
&=\mathbb{E}^\mu\left[U_{\alpha_k}f_k(X_{T-}) U_{\alpha_l}f_l(X_{T-})\right].
\end{align*}
By the assumed separating property of $\left(U_{\alpha_k}f_k\right)_{k\geq 1}$ we obtain that for all $\varphi\in b\mathcal{B}(E)$ and $l\geq 1$
\begin{equation*}
\mathbb{E}^\mu\left[\varphi(X_{T}) U_{\alpha_l}f_l(X_{T-})\right]=\mathbb{E}^\mu\left[\varphi(X_{T-}) U_{\alpha_l}f_l(X_{T-})\right].
\end{equation*}
Arbitrarily fixing $\varphi$ in the above inequality and using again the separating property of $\left(U_{\alpha_l}f_l\right)_{l\geq 1}$ we deduce that for all $\varphi,\psi\in b\mathcal{B}(E)$
\begin{equation*}
\mathbb{E}^\mu\left[\varphi(X_{T}) \psi(X_{T-})\right]=\mathbb{E}^\mu\left[\varphi(X_{T-}) \psi(X_{T-})\right],
\end{equation*}
and by a monotone class argument that
\begin{equation*}
\mathbb{E}^\mu\left[h(X_{T},X_{T-})\right]=\mathbb{E}^\mu\left[h(X_{T-},X_{T-})\right], \quad h\in b\mathcal{B}(E\times E).
\end{equation*}
Taking $h:=1_{{\sf diag}(E\times E)}$ we obtain that $\mathbb{P}^\mu(X_{T}=X_{T-})=1$, so Step II is done.

{\bf Step III:} {\it $X$ is quasi-left continuous with respect any quasi-$\mathcal{U}$ toplogy}.
Let $\overline{\tau}$ be a quasi-$\mathcal{U}$-topology on $E$ with attached $\mathcal{U}$-nest $(F_n)_{n\geq 1}$, as in Definition \ref{def:quasi-U}, and $T_n,n\geq 1$ be $\left(\mathcal{F}_t\right)$-stopping times with $T_n\nearrow T$.
Fix $\mu\in \mathcal{P}(E)$.
Since $[T<\infty]=\cup_{n\geq 1} [T<T_{F_k^c}]$ $\mathbb{P}^\mu$-a.s., it is sufficient to prove that $\lim\limits_n X_{T_n}=X_T$ $\mathbb{P}^\mu$-a.s. w.r.t. $\overline{\tau}$ on $T<T_{F_k^c}$ for all $k\geq 1$.
Since $\overline{\tau}|_{F_k}$ is the trace on $F_k$ of the topology generated by the countable family $U_{\alpha_k}f_k, k\geq 1$ which separates the points,
it is sufficient to show that if we generically set $u:=U_{\alpha_k}f_k$, then
\begin{equation} \label{eq:leftc}
    \lim\limits_n u(X_{T_n})=u(X_T) \mbox{ on } [T<\infty]  \quad \mathbb{P}^\mu\mbox{-a.s. for all } \mu\in\mathcal{P}(E).
\end{equation}
To see this, note first that as in Step II, we can assume without loss of generality that $T$ as well as $\sup\limits_{t\in [0,T]}U_{\alpha_k}|f_k|(X_t)$ are bounded.
Further, by Step II we have that $X$ is quasi-left continuous with respect to $\tau$, which implies that $X_T$ is $\bigvee\limits_n \mathcal{F}_{T_n}^\mu$-measurable, where $\bigvee\limits_n \mathcal{F}_{T_n}^\mu$ is the $\sigma$-algebra generated by $\bigcup\limits_n \mathcal{F}_{T_n}^\mu$.
As a matter of fact, $\bigvee\limits_n \mathcal{F}_{T_n}^\mu=\left(\bigvee\limits_n \mathcal{F}_{T_n}\right)^\mu$.
Then, since
\begin{equation}
    M_t:=u(X_t)-u(X_0)-\int_0^t\alpha_k(\alpha_k U_{\alpha_k} f_k-f_k)(X_s) \; ds, \quad 0\leq t\leq T
\end{equation}
is a c\`adl\`ag $(\mathcal{F}_t)_{t\geq 0}$-martingale under $\mathbb{P}^\mu$, by Doob's stopping theorem we get that
\begin{equation} \label{eq:mart}
    M_{T_n}=\mathbb{E}^\mu\left[M_T | \mathcal{F}_{T_n}\right], \quad n\geq 1, \mathbb{P}^\mu\mbox{-a.s}.
\end{equation}
On the other hand, by the martingale convergence theorem and the fact that $M_T$ is also $\bigvee\limits_n \mathcal{F}_{T_n}^\mu$-measurable, we have that
\begin{equation} \label{eq:conv}
    \lim\limits_n M_{T_n}= \mathbb{E}^\mu\left[M_T \bigg| \bigvee\limits_n \mathcal{F}_{T_n}\right]=\mathbb{E}^\mu\left[M_T \bigg| \left(\bigvee\limits_n \mathcal{F}_{T_n}\right)^\mu\right]= \mathbb{E}^\mu\left[M_T \bigg| \bigvee\limits_n \mathcal{F}_{T_n}^\mu\right]=M_T  \quad \mathbb{P}^\mu\mbox{-a.s}. 
\end{equation}
This implies that \eqref{eq:leftc} holds and thus Step III is complete.

\medskip
{\bf Step IV: }{\it For all $\mu\in\mathcal{P}(E)$ we have that $X$ has $\mathbb{P}^\mu$-a.s. c\`adl\`ag paths with respect to any quasi-$\mathcal{U}$-topology on $E$.} 
To prove this final claim, let $\overline{\tau}$ be a quasi-$\mathcal{U}$-topology on $E$ with attached $\mathcal{U}$-nest $(F_n)_{n\geq 1}$, and $U_{\alpha_k}f_k$ as in \Cref{def:quasi-U}.
Also, let $\tilde{\tau}$ be the topology on $E$ generated by $U_{\alpha_k}f_k, k\geq 1$.
Note that $\tilde{\tau}$ is a natural topology on $E$ with respect to $\mathcal{U}$.
Moreover, since $\overline{\tau}|_{F_n}=\tilde{\tau}|_{F_n}$ for each $n\geq 1$, and $(F_n)_{n\geq 1}$ forms a $\mathcal{U}$-nest, we obtain that $X$ is quasi-left continuous with respect to $\overline{\tau}$ as well. 
This means that $X$ is a standard process in the sense of \cite[Definition 47.4]{Sh88}.
Since $X$ is conservative, by employing \cite[Theorem 47.10]{Sh88}, we obtain that $X$ has left limits in $E$ with respect to $\overline{\tau}$ $\mathbb{P}^\mu$-a.s. for all $\mu\in \mathcal{P}(E)$.
Now we use the localization with respect to the $\mathcal{U}$-nest $(F_n)_{n\geq 1}$ once more, but this time using that the $F_n$'s are $\tilde\tau$-closed, to deduce that $X$ has left limits (hence is c\`adl\`ag) with respect to $\tilde{\tau}$ $\mathbb{P}^\mu$-a.s. for all $\mu\in \mathcal{P}(E)$.

\medskip
Clearly, putting the last two steps together we obtain that $X$ is a Hunt process with respect to any quasi-$\mathcal{U}$-topology on $E$.
\end{proof}

\begin{rem}\label{rem:cadlag+Feller=Hunt}
\begin{enumerate}
    \item[(i)] As mentioned, the fundamental role of \Cref{prop:quasi_left} in this work is to provide the necessary tool to deduce that a right process has c\`adl\`ag paths in some given topology, by proving the same property in some weaker topology. 
    This kind of argument stands behind the main existence results obtained in \Cref{S:Hunt} for generalized Mehler semigroups defined on a Hilbert space, where the corresponding process is essentially shown to have c\`adl\`ag paths with respect to the weak topology, but then this property is automatically transfered to the norm topology since both topologies are shown to be quasi-$\mathcal{U}$ topologies, or weaker.
    \item[(ii)] By a straightforward adaptation of Step I and Step II in the proof of \Cref{prop:quasi_left}, one can also deduce the following essentially known result:  If $X$ is a conservative right process on $E$ which is a.s. c\`adl\`ag with respect to some natural topology $\tau$, and moreover its semigroup $(P_t)_{t\geq 0}$ has the Feller property with respect to $\tau$, then $X$ is a Hunt process with respect to $\tau$.
\end{enumerate}
\end{rem}

\section{Construction of Hunt processes}\label{S:Hunt}

Let $\mathcal{U}$ be a Markovian resolvent of probability kernels on the Lusin measurable space $(E,\mathcal{B})$, endowed with a given topology $\tau$ such that $\sigma(\tau)=\mathcal{B}$.

\begin{defi} \label{defi 3.1}
Let $V\in\mathcal{E}(\mathcal{U}_\alpha)$ for some $\alpha \geq 0$.
\begin{enumerate}[(i)]
    \item We say that $V$ is a norm-like (sequentially) $\tau$-compact $\alpha$-Lyapunov function if $U_{\alpha+1}V<\infty$, i.e. $[V=\infty]$ is polar according to \Cref{rem:polar_characterization}, and the level sets $[V\leq n]$ are (sequentially) relatively $\tau$-compact for all $n\geq 1$.
    \item We say that $V$ is a (sequentially) $\tau$-compact $\alpha$-Lyapunov function vanishing at infinity if $V>0$ and the level sets $\left[V\geq\dfrac{1}{n}\right]$ are (sequentially) relatively $\tau$-compact for all $n\geq 1$.
\end{enumerate}
If $V$ is in any of the above two cases, we call it a (sequentially) $\tau$-compact $\alpha$-Lyapunov function.
\end{defi}

\begin{rem}\label{rem:lyapunov_beta}
$(i)$ Clearly, if $V$ is a ($\sigma$)$\tau$-compact $\alpha$-Lyapunov function for some $\alpha\geq 0$, then it remains a ($\sigma$)$\tau$-compact $\beta$-Lyapunov function for all $\beta\geq \alpha$.

$(ii)$ For examples of using $\tau$-compact $\alpha$-Lyapunov functions 
(where $\tau$ is a metrizable Lusin topology) in constructing c\`adl\`ag right process on various infinite dimensional spaces see 
\cite{Be11}, 
\cite{BeLu 16}, \cite{BeCoRo11}, \cite{BeRo11a}, and the survey paper \cite{BeRo11b}. 

\end{rem}

Now we introduce the main assumption that guarantees the existence of a regular Markov process. 
In fact, there are two assumptions $\mathbf{H_{V,\tau}}$ and $\mathbf{H_{V,\sigma\tau}}$ introduced simultaneously, for simplicity:

\medskip
\noindent
$\mathbf{H_{V,(\sigma)\tau}}$. Assumption $\mathbf{H_{\mathcal{C}}}$ holds and there exist $\beta>\alpha\geq 0$, a (sequentially) $\tau$-compact $\alpha$-Lyapunov function $V$ and a family $(f_k)_{k\geq 1}$ of real-valued $\mathcal{B}$-measurable functions such that
\begin{enumerate}[(i)]
    \item If $V$ is norm-like then $|f_k|\leq c_1 V+c_2$, $k\geq 1$, where $c_1,c_2$ are constants that can depend on $k$, whilst if $V$ is vanishing at infinity then $f_k\in \mathcal{B}_b(E), k\geq 1$.
    \item The family $(f_k)_{k\geq 1}$ separates the set of all finite measures $\nu$ on $E$ which satisfy $\nu(|f_k|)<\infty, k\geq 1$.
    \item For all $n,k\geq 1$ we have that $U_\beta f_k$ is (sequentially) $\tau$-continuous on $\overline{[V\leq n]}$ if $V$ is a norm-like $\alpha$-Lyapunov function, respectively $U_\beta f_k$ is (sequentially) $\tau$-continuous on $\overline{[V\geq 1/n]}$ if $V$ is vanishing at infinity; here, the closure is taken with respect to $\tau$.
\end{enumerate} 

\begin{prop} \label{prop:V_quasy_ray}
If $\mathbf{H_{V,\tau}}$ is satisfied, then the restriction of $\tau$ on $\overline{[V\leq n]}$ if $V$ is norm-like (resp. on $\overline{[V\geq 1/n]}$ if $V$ is vanishing at infinity) is metrizable for all $n\geq 1$, and moreover, $\tau$ is a quasi-$\mathcal{U}$-topology with attached $\mathcal{U}$-nest $(\overline{[V\leq n]})_{n\geq 1}$ (resp. $(\overline{[V\geq 1/n]})_{n\geq 1}$).
\end{prop}
\begin{proof}
We treat only the case when $V$ is a norm-like $\tau$-compact $\alpha$-Lyapunov function, since the other case is similar.
Let $n\geq 1$. 
Since the family of $\tau$-continuous functions $(U_\beta f_k)_{k\geq 1}$ separates the points of the $\tau$-compact set $\overline{[V\leq n]}$ we get that $\overline{[V\leq n]}$ is Hausdorff, hence by Stone-Weierestrass theorem the family $(U_\beta f_k)_{k\geq 1}$ generates the topology $\tau$ restricted to $\overline{[V\leq n]}$. 
In particular, $\overline{[V\leq n]}$ is also second-countable, so by Urysohn's metrization theorem we conclude that the restriction of $\tau$ on $\overline{[V\leq n]}$ is metrizable. 

The fact that $\tau$ is a quasy-$\mathcal{U}$ topology with attached $\mathcal{U}$-nest $(\overline{[V\leq n]})_{n\geq 1}$ is now clear since according to Proposition \ref{prop:nest} $([V\leq n])_{n\geq 1}$, and hence $(\overline{[V\leq n]})_{n\geq 1}$, are $\mathcal{U}$-nests, and since $\overline{[V\leq n]}$ is $\tau$-compact it is also closed in the topology generated by $\left(U_\beta f_k\right)_{k\geq 1}$.
\end{proof}

\begin{thm}[Construction of Hunt processes]\label{thm:main_construction}
Let $\mathcal{U}$ be a Markovian resolvent of probability kernels on a Lusin measurable space $(E,\mathcal{B})$, and $\tau$ be a given topology on $E$ such that $\sigma(\tau)=\mathcal{B}$.
If $\mathbf{H_{V,\sigma\tau}}$ is fulfilled then the following assertions hold.
\begin{enumerate}[(i)]
    \item There exists a right process $X$ on $E$ with resolvent $\mathcal{U}$, hence it is c\`ad w.r.t. $\tau$.
    If the stronger assumption $\mathbf{H_{V,\tau}}$ is satisfied, then $X$ is c\`adl\`ag with respect to $\tau$.
    \item If $\tau$ is finer than a quasi-$\mathcal{U}$-topology on $E$ and $X$ is c\`adl\`ag with respect to $\tau$, then $X$ is a Hunt process with respect to any quasi-$\mathcal{U}$-topology on $E$.
    \item If the compact Lyapunov function $V$ appearing in $\mathbf{H_{V,\sigma\tau}}$ is vanishing at infinity (see \Cref{defi 3.1}), then $X$ is saturated (see \Cref{defi:saturated}).
\end{enumerate}
\end{thm}

\begin{proof}
Let $V$ be a sequentially $\tau$-compact $\alpha$-Lyapunov function so that $\mathbf{H_{V,\sigma\tau}}$ holds with $\beta>\alpha$ as in $\mathbf{H_{V,\sigma\tau}}$, (iii). 
Further, let $E_1$ be the saturation of $E$ w.r.t. $\mathcal{U}_\beta$ and $\mathcal{U}^1$ the extension of $\mathcal{U}$ from $E$ to $E_1$ given by \Cref{defi:saturation}, and $\tilde{V}$ be the unique $\mathcal{U}^1_\alpha$-excessive extension of $V$ from $E$ to $E^1$.
Further, we proceed in two steps.

\medskip
\noindent
{\bf Proof of (ii).} The statement follows directly from \Cref{prop:quasi_left}.

\medskip
\noindent
{\bf Proof of (i) and (iii).} 
We treat the two cases separately.

\medskip
\noindent
{\bf Existence of $X$ when $V$ is norm-like.} 
In this case, it is enough to show that $E_1\setminus E=[\tilde{V}=\infty]$.
Indeed, if the latter equality holds and since $U_1^1(1_{E_1\setminus E})\equiv 0$, we get by \Cref{rem:polar_characterization} that $E_1\setminus E$ is polar, and the existence of a right Markov process on $E$ with resolvent $\mathcal{U}$ follows by \Cref{thm 4.15}.

To show that $E_1\setminus E=[\tilde{V}=\infty]$, let $\xi\in E_1$ such that $\tilde{V}(\xi)<\infty$.
Let $\tau_{\mathcal{R}^1}$ be the Ray topology on $E_1$ generated by the unique exten Ray cone $\mathcal{R}_1$ uniquely induced by $\mathcal{R}_0:=\left\{ n\wedge U_\beta f_k^\pm \; : \; k,n\geq 1\right\}\cup \left\{ V\wedge n \; : \; n\geq 1\right\}$, as in \Cref{rem: ray_topology}; here, $f^+$ and $f^-$ are the positive, resp. negative part of $f$.
In particular, $U_\beta f_k^{\pm}$, $k\geq 1$ are all $\tau_{\mathcal{R}^1}$-continuous, as well as $\tilde V$.
Because $E$ is dense in $E_1$ w.r.t. the fine topology, it is also dense with respect to any coarser topology, in particular,  with respect to $\tau_{\mathcal{R}^1}$, the latter being metrizable.
Therefore, we can find a sequence $(x_n)_{n\geq 1}\subset E$ which converges to $\xi$ with respect to $\tau_{\mathcal{R}^1}$, so that
\begin{equation} \label{eq:limit_xi}
  \lim\limits_n U_\beta f_k(x_n)=U_\beta^1 f_k(\xi),\; k\geq 1, \quad \mbox{ and } \quad  \lim\limits_nV(x_n)=\tilde V(\xi)<\infty,  
\end{equation}
where $f_k$ is extended from $E$ to $E_1$ by setting e.g. $f_k\equiv 0$ on $E_1\setminus E$.
In particular, there exists $N<\infty$ such that $(x_n)_{n\geq 1}\subseteq [V\leq N]$.
We emphasize that by $\mathbf{H_{V,\tau}}$, (i), we have that $U_\beta^1 f_k(\eta)$ is well defined and finite for all $\eta \in E_1$ for which $\tilde V(\eta)<\infty$. 
Now, since $[V\leq N]$ is sequentially $\tau$-compact, there exists $x^\ast \in E$ and a subsequence $(x_{n_k})_{k\geq 1}$ such that $\tau-\lim\limits_k x_{n_k}=x^\ast$.
Thus, since $U_\beta f_k$ is sequentially $\tau$-continuous on $\overline{[V\leq N]}$, we get that $\lim\limits_n U_\beta f_k(x_n)=U_\beta f_k(x^\ast)$ for each $k\geq 1$.
On the other hand, the first convergence in \eqref{eq:limit_xi} leads to
\begin{equation*}
U_\beta^1 f_k(\xi)=U_\beta f_k(x^\ast) \mbox{ for all } k \geq 1,    
\end{equation*}
hence, by \eqref{eq 4.1}, $\xi(f_k)=\delta_{x^\ast}\circ U_\beta(f_k)$ for all $k\geq 1$. 
%On the other hand, 
By $\mathbf{H_{V,\sigma\tau}}$, (i) we have that $\xi(|f_k|)+\delta_{x^\ast}\circ U_\beta(|f_k|)<\infty$,  {so,}   
by 
$\mathbf{H_{V,\sigma\tau}}$, (ii)  $(f_k)_{k\geq 1}$ separates $\xi$ from $\delta_{x^\ast}\circ U_\beta$, hence $\delta_{x^\ast}\circ U_\beta=\xi$ which means that $\xi \in E$. 
In conclusion $E_1\setminus E=[\tilde{V}=\infty]$.

\medskip
\noindent
{\bf Existence of $X$ and the saturation property when $V$ is vanishing at infinity.} 
By \Cref{thm 4.15} and \Cref{thm:saturation}, in order to prove the existence of a right process which is also saturated, it is sufficient (and necessary) to show that $E_1\setminus E=\emptyset$, which in turn will follow immediately if we show that $E^1\setminus E=[\tilde V=0]$. 
Indeed, by \Cref{prop:zero_open} we have that $[\tilde V=0]$ is finely open with respect to the fine topology generated by $\mathcal{U}^1$. 
Also, $U_{1}^1(1_{[E^1\setminus E]})\equiv 0$, hence $E_1\setminus E=\emptyset$; here we use the potential theoretic fact that a finely open set with null potential must be empty.

Now, let us show that we have $E^1\setminus E=[\tilde V=0]$ indeed.
Let $\tau_{\mathcal{R}^1}$ be the Ray topology constructed in the previous case, and $\xi\in E_1\setminus E$ be such that $\tilde V(\xi)>0$. 
Now, the proof follows similarly to the previous case. 
Namely, we can find a sequence $(x_n)_{n\geq 1}\subset E$ which converges to $\xi$ with respect to $\tau_{\mathcal{R}^1}$, so that
\begin{equation} \label{eq:limit_xi_2}
  \lim\limits_n U_\alpha f_k(x_n)=U_\beta^1 f_k(\xi),\; k\geq 1, \quad \mbox{ and } \quad  \lim\limits_nV(x_n)=\tilde V(\xi)>0. 
\end{equation}
Note that $U_\beta^1 f_k(\xi)$ is well defined and bounded for all $\xi \in E_1$, since in this case $f_k$ is assumed to be bounded.
From \eqref{eq:limit_xi_2} we in particular deduce that there exists $N<\infty$ such that $(x_n)_{n\geq 1}\subseteq [V> 1/N]$. 
Since $[V> 1/N]$ is sequentially $\tau$-compact, there exists $x^\ast \in E$ and a subsequence $(x_{n_k})_{k\geq 1}$ such that $\tau-\lim\limits_k x_{n_k}=x^\ast$.
Therefore, since $U_\beta f_k$ is sequentially $\tau$-continuous on $\overline{[V> 1/N]}$, on the one hand we get that $\lim\limits_n U_\beta f_k(x_n)=U_\beta f_k(x^\ast)$ for each $k\geq 1$.
On the other hand, the first convergence in \eqref{eq:limit_xi_2} leads to
\begin{equation*}
U_\beta^1 f_k(\xi)=U_\beta f_k(x^\ast) \mbox{ for all } k \geq 1,    
\end{equation*}
hence, by \eqref{eq 4.1}, $\xi(f_k)=\delta_{x^\ast}\circ U_\beta(f_k)$ for all $k\geq 1$. 
On the other hand, by $\mathbf{H_{V,\sigma \tau}}$, (i) we have that $\xi(|f_k|)+\delta_{x^\ast}\circ U_\alpha(|f_k|)<\infty$, so by $\mathbf{H_{V,\sigma \tau}}$, (ii)  $(f_k)_{k\geq 1}$ separates $\xi$ from $\delta_{x^\ast}\circ U_\beta$, hence $\delta_{x^\ast}\circ U_\beta=\xi$ which means that $\xi \in E$. 
In conclusion, $E_1\setminus E=[\tilde V=0]$.

\medskip
\noindent
{\bf $\mathbf X$ has $\mathbf \tau$-c\`adl\`ag paths.}
Assume that $\mathbf{H_{V,\tau}}$ is satisfied, hence $[V\leq n]$ is relatively $\tau$-compact and $U_\beta f_k$ is $\tau$-continuous on each $\overline{[V\leq n]}$, $n,k\geq 1$, and let $X$ be a right Markov process on $E$ with resolvent $\mathcal{U}$, e.g. the one constructed in (i).
First of all, by \Cref{prop:V_quasy_ray} and \Cref{coro:right_quasi_ray}, $X$ has $\mathbb{P}^\mu$-a.s. right continuous paths with respect to $\tau$ for all $\mu\in \mathcal{P}(E)$.
So it remains to prove that $X$ has left limits with respect to $\tau$.
We do this by localization, so let us set $F_n:=\overline{[V\leq n]}$ if $V$ is norm-like, and $F_n:=\overline{[V\geq 1/n]}$ if $V$ is vanishing at infinity, $n\geq 1$.
By \Cref{rem:prob_nest} and \Cref{prop:nest}, we have that $\lim_{n\to \infty}T_{E\setminus F_n} = \infty$ $\mathbb{P}^\mu$-a.s. for all $\mu\in \mathcal{P}(E)$.
Consequently, it is sufficient to show that $\mathbb{P}^\mu$-a.s. $X$ has paths with left limits in $E$ with respect to $\tau$ on each time interval $\left[0,T_{E\setminus F_n}\right)$, $n\geq 1$, $\mu\in \mathcal{P}(E)$.
Since on each $\left[0,T_{E\setminus F_n}\right)$ the process lies in $F_n$, and the family of $\tau|_{F_n}$-continuous functions $U_\beta f_k, k\geq 1$ separates the points of the $\tau$-compact set $F_n$, it is enough to show that $[0,T_{F_n})\ni t \mapsto U_\beta f_k(X_t)\in \mathbb{R}$ has left limits $\mathbb{P}^\mu$-a.s. for all $\mu\in \mathcal{P}(E)$. 
But the latter property is now clear since $e^{-\beta t}U_\beta f_k(X_t)=e^{-\beta t}U_\beta f_k^+(X_t)-e^{-\beta t}U_\beta f_k^-(X_t), t\geq 0$ is the difference of two c\`adl\`ag supermartingales under $\mathbb{P}^\mu$ for all $\mu\in \mathcal{P}(E)$.
\end{proof}

\paragraph{Saturated Hunt processes associated to weakly Feller semigroups on locally compact spaces.}
In this short paragraph we show that the previously obtained general result, namely \Cref{thm:main_construction}, can be easily used to enhance even the well established result that guarantees the existence of Hunt process associated to a Feller (Markov) semigroup on locally compact spaces. For the sake of comparison, we recall this well known result here, cf. e.g. \cite[Theorem 9.4]{BlGe68}:
{\it
Let $E$ be a locally compact space with a countable base, $\mathcal{B}$ be its Borel $\sigma$-algebra, and $C_0(E)$ denote the space of real-valued continuous functions on $E$ which vanish at infinity.
Let $(P_t)_{t\geq 0}$ be a semigroup of Markov kernels $(E,\mathcal{B})$ which is Feller, i.e. i) $P_0=Id$, \; ii) $P_t(C_0(E))\subset C_0(E), t>0$, \; and iii) $P_tf \mathop{\rightarrow}\limits_{t\to 0} f$ uniformly on $E$ for every $f\in C_0(E)$.
Then there exists a Hunt process on $E$ with transition function $(P_t)_{t\geq 0}$.
}

Note that under the assumptions of the classical result from above, for every $0<v\in C_0(E)$ we have that $V:=U_qv \in C_0(E)$ is a compact Lyapunov function vanishing at infinity, in the sense of \Cref{defi 3.1}.
The refined version of the above result, derived directly from \Cref{thm:main_construction}, is the following:

\begin{coro} \label{coro:lccb}
Let $E$ be a locally compact space with a countable base, $\mathcal{B}$ be its Borel $\sigma$-algebra, and $(P_t)_{t\geq 0}$ be a normal semigroup of Markov kernels $(E,\mathcal{B})$ which have the (weak) Feller property, i.e. i) $P_0=Id$, \; ii) $P_t(C_b(E))\subset C_b(E), t>0$, \; and iii) $P_tf \mathop{\rightarrow}\limits_{t\to 0} f$ pointwise on $E$ for every $f\in C_b(E)$.
Further, let $\mathcal{U}:=\left(U_\alpha\right)_{\alpha>0}$ be the resolvent of $(P_t)_{t\geq 0}.$
If there exists a function $\mathcal{B}_b\ni v>0$ and $\beta>0$ such that $U_\beta v\in C_0(E)$, then there exists a Hunt process on $E$ with transition function $(P_t)_{t\geq 0}$, and any such process is saturated.    
\end{coro}

\begin{proof}
First of all, it is clear that $\mathcal{U}$ satisfies condition $(\mathbf{H_\mathcal{C}})$ with $\mathcal{C}=C_b(E)$.
Further, let $\mathcal{A}_0:=(f_n)_{n\geq 0}\subset C_b(E)$ be measure separating, $\beta>0$, and denote by $\mathcal{R}_{\mathcal{A}_0}$ the Ray cone generated by $\mathcal{A}_0$ associated to $\mathcal{U}_\beta$, as described in \Cref{rem: ray_topology}, (ii); let $\tau_{\mathcal{R}_{\mathcal{A}_0}}$ be the corresponding Ray topology.
On the one hand, it is clear that $\mathcal{R}_{\mathcal{A}_0}\subset C_b(E)$, hence $\tau_{\mathcal{R}_{\mathcal{A}_0}}$ is coarser than the original topology on $E$. 
On the other hand, the original topology on $E$ is fully determined by its trace on the compact sets, and on each such compact set, by Stone-Weierstrass theorem, $\mathcal{A}_0$ generates it. 
Since $\mathcal{A}_0\subset \mathcal{R}_{\mathcal{A}_0}$, we obtain that the original topology is coarser than $\tau_{\mathcal{R}_{\mathcal{A}_0}}$.
Consequently, the original topology coincides with $\tau_{\mathcal{R}_{\mathcal{A}_0}}$.

Further, let $C_0(E)\ni V:= U_\beta v>0$, so that $V$ is a compact $\beta$-Lyapunov function vanishing at infinity.
It is now clear that condition $\mathbf{H_{V,\tau_{\mathcal{R}_{\mathcal{A}_0}}}}$ is completely fulfilled, so the result follows by \Cref{thm:main_construction}.
\end{proof}

\noindent
{\bf Example.} Any L\'evy process on $\mathbb{R}^d$ (i.e. a Hunt process whose semigroup is of convolution type) is a saturated process. 
In the next section we shall prove that the saturation property remains valid even for {\it generalized Mehler semigroups} defined on Hilbert spaces, in the case when the drift is given by a bounded operator. However, this will no longer hold when the drift is unbounded.

\section{Generalized Mehler semigroups and construction of associated Markov processes}\label{S:Mehler}

In this section we follow the framework of \cite{BoRoSc96}, \cite{FuRo00}, or \cite{LeRo04}. Throughout, $(H, \langle \cdot, \cdot\rangle)$ is a real separable Hilbert space with Borel $\sigma$-algebra denoted by $\mathcal{B}$, in particular $(H,\mathcal{B})$ is a Lusin measurable space.
Also let us fix
\begin{enumerate}
    \item[-] a strongly continuous semigroup $(T_t)_{t\geq 0}$ on $H$ with (possibly unbounded) generator $\left({\sf A}, D({\sf A})\right)$ and
    \item[-] a family of probability measures $(\mu_t)_{t\geq 0}$ on $H$ such that 
    \begin{enumerate}
    \item[($\rm M_1$)] $\lim\limits_{t\to 0}\mu_t = \mu_0=\delta_0$ weakly,
    \item[($\rm M_2$)] $\mu_{t+s}=(\mu_t\circ T_s^{-1})\ast \mu_s$ for all $s,t\geq 0$.
\end{enumerate}
\end{enumerate} 
Then we can consider the {\it generalized Mehler semigroup} $(P_t)_{t\geq 0}$ on $H$ defined as
\begin{equation}\label{eq: Mehler}
P_tf(x)=\int_H f(T_tx+y) \; \mu_t(dy) \quad \mbox{ for all } x\in H \mbox{ and } f\in b\mathcal{B}(H). 
\end{equation}
Clearly, $(P_t)_{t\geq 0}$ is a semigroup of Markov operators on $H$. 
Our main goal here is to show that under some conditions that generalize the usual ones used to construct Ornstein-Uhlembeck processes or stochastic convolutions with continuous paths in infinite dimensions, we can always associate a Hunt process with  $(P_t)_{t\geq 0}$,  with state space $H$.

Since $(T_t)_{t\geq 0}$ is strongly continuous on $H$, there exist $M,\omega\geq 0$ such that
\begin{equation} \label{eq 1.1}
|T_tx|_H\leq Me^{\omega t}|x|_H \quad \mbox{ for all } x\in H \mbox{ and } t\geq 0.
\end{equation}

Now, let us collect some regularity properties of $(P_t)_{t\geq 0}$ that result easily from \Cref{eq: Mehler}. 
To this end, let us denote by $C_b^w(H)$ (resp. $C_b^{\sigma w}(H)$) the space of bounded (resp. sequentially) weakly continuous functions on $H$.

\begin{prop} \label{prop: regularity_Mehler}
The following assertions hold for all $t>0$
\begin{enumerate}[(i)]
    \item The mapping $[0,\infty)\times H\ni(t,x) \mapsto P_tf(x)\in \mathbb{R}$ is continuous for every  $f\in Lip_b(H)$. 
    Moreover, $P_t(Lip_b(H))\subset Lip_b(H)$ and $P_t(C_b(H))\subset C_b(H)$ for all $t> 0$.
    \item $P_t(C_b^{\sigma w}(H))\subset C_b^{\sigma w}(H)$.
    \item If $T_t$ is compact, then $P_t(C_b(H))\subset C_b^{\sigma w}(H)$.
\end{enumerate}
\end{prop}

\begin{proof}
i). If $f\in Lip_b(H)$ and $x,x'\in H$, and $t,t'\geq 0$ then
\begin{align*}
    \left|P_tf(x)-P_{t'}f(x')\right|&\leq \left|P_tf(x)-P_{t}f(x')\right|+\left|P_{t}f(x')-P_{t'}f(x')\right|\\
    &\leq M|f|_{Lip} e^{\omega_0t}|x-x'|_H+\left|\int_H f(T_{t}x'+y) \; \mu_t(dy)-\int_H f(T_{t'}x'+y) \; \mu_{t'}(dy)\right|\\
    &\leq  M|f|_{Lip} e^{\omega_0t}|x-x'|_H+\left|\int_H f(T_{t}x'+y) \; \mu_t(dy)-\int_H f(T_{t'}x'+y) \; \mu_{t}(dy)\right|\\
    &\quad +\left|\int_H f(T_{t'}x'+y) \; \mu_t(dy)-\int_H f(T_{t'}x'+y) \; \mu_{t'}(dy)\right|\\
    &\leq M|f|_{Lip} e^{\omega_0t}|x-x'|_H+M|f|_{Lip} e^{\omega_0t}|T_t x'-T_{t'}x'|_H\\
    &\quad + \left|\int_H f(T_{t'}x'+y) \; \mu_t(dy)-\int_H f(T_{t'}x'+y) \; \mu_{t'}(dy)\right|
\end{align*}
which converges to $0$ when $(t,x)$ convergences to $(t',x')$ by the strong continuity of $(T_t)_{t\geq 0}$ and the weak continuity of $(\mu_t)_{t\geq 0}$.

The fact that $P_t(Lip_b(H))\subset Lip_b(H)$ is clear by taking $t=t'$ in the above inequality, whilst $P_t(C_b(H))\subset C_b(H)$ follows simply by dominated convergence, or, if we wish, as a consequence of the former inclusion, since the set of functions $Lip_b(H)$ is convergence determining for the weak convergence of probability measures on $H$.

ii). If $x_n \mathop{\rightharpoonup}\limits_n  x$ in $H$ then $T_tx_n \mathop{\rightharpoonup}\limits_n  T_t x$ for each $t\geq 0$. 
Hence, if $f\in C_b^{\sigma w}(H)$ then $\lim\limits_nf(T_tx_n+y)= f(T_tx+y)$ for all $y\in H$, so the assertion follows again by dominated convergence.

iii). This assertion is proved similarly to ii), since a compact operator maps weakly convergent sequences to strongly convergent ones.
\end{proof}
By \Cref{prop: regularity_Mehler}, we  can construct the resolvent of Markov kernels on $H$ associated with $(P_t)_{t\geq 0}$, which we denote by $\mathcal{U}:=(U_{\alpha})_{\alpha > 0}$.
Further, $\tau_{\|\|}$ (resp. $\tau_w$) denotes the norm (resp. weak) topology on $H$, whilst $({\sf L}, D_b({\sf L}))$ stands for the weak generator associated with  $\mathcal{U}$ as given by \Cref{eq:D_b_L}.

In the next propositions we analyze the existence of compact Lyapunov functions for the generalized Mehler semigroup $(P_t)_{t\geq 0}$ given by \eqref{eq: Mehler}, depending on whether $A$ is a bounded or an unbounded operator, respectively; also, we investigate if the norm and weak topologies are $\mathcal{U}$ or quasi-$\mathcal{U}$ topologies on $H$; recall \Cref{def:quasi-U}.

\begin{prop}\label{prop:tau_||}
If the generator ${\sf A}$ of $(T_t)_{t\geq 0}$ is bounded, then the following assertions hold:
\begin{enumerate}[(i)]
    \item For all $T> 0$ we have
        \begin{equation}
        \lim_{\|x\|\to \infty} \inf_{t\in [0,T]}\|T_tx\|=\infty.
        \end{equation}
    \item $V:H \rightarrow [0,1]$ given by
    \begin{equation}\label{eq:V_0}
        V(x):= U_1v(x), \quad  \mbox{ where } v(x):= (1+\|x\|)^{-1}, \quad x\in H,
    \end{equation}   
    is a $\tau_w$-compact $1$-Lyapunov function vanishing at infinity.
    \item If $f\in Lip_b(H)$ has bounded support, then
    \begin{equation}\label{eq:uniform_A_bounded}
    \lim\limits_{\alpha\to \infty} |\alpha U_\alpha f-f|_\infty=0.
    \end{equation}
    \item $\tau_{\|\|}$ is a $\mathcal{U}$-topology on $H$.
\end{enumerate}
\end{prop}
\begin{proof}
(i). Since $A$ is bounded we have
\begin{align}
    &\|T_tx-x\|\leq \left(e^{t\|A\|}-1\right)\|x\|\leq t\|A\|e^{\|A\|}\|x\|\leq \|x\|/2 \quad \mbox{ for } t\leq t_0:= \left(2\|A\|e^{\|A\|}\right)^{-1},\\
    & \|T_t x\| \geq \|x\|-\|T_tx-x\|\geq \|x\|/2 \quad \mbox{ for } t\leq t_0,\\
    & \|T_t x\|=\|T_{t-[t/t_0]t_0}T_{t_0}^{[t/t_0]}x\|\geq 2^{-([t/t_0]+1)}\|x\| \quad \mbox{ for all } t>0,
\end{align}
so assertion (i) follows.

(ii). Notice that $0<V\leq 1$ is $\mathcal{U}_1$-excessive, and by \Cref{prop: regularity_Mehler} it is also $\tau_{\|\|}$-continuous.
Furthermore, by assertion (i) and dominated convergence we have
    \begin{equation}
        \lim_{\|x\|\to \infty} V(x)= \int_0^\infty e^{-t} \int_H \lim_{\|x\|\to \infty}\left(1+\|T_tx+y\|\right)^{-1} \mu_t(dy)\;dt=0,
    \end{equation}
which means that $[V\geq 1/N]$ is $\tau_w$-compact for all $N\geq 1$, hence $V$ is a $\tau_w$-compact 1-Lyapunov function vanishing at infinity.

(iii).
If $f\in Lip_b(H)$ then
\begin{align}
    \sup\limits_{\|x\|\leq R}|\alpha U_\alpha f(x)-f(x)|&\leq \sup\limits_{\|x\|\leq R}\alpha\int_0^\infty e^{-\alpha t}\int_H |f(T_tx+y)-f(x)|\;\mu_t(dy)\;dt\\
    &\leq |f|_{Lip}\left[\alpha\int_0^\infty e^{-\alpha t}\sup\limits_{\|x\|\leq R}\|T_tx-x\|\;dt+\alpha\int_0^\infty e^{-\alpha t}\int_H \|y\|\wedge(2|f|_\infty)\;\mu_t(dy)\;dt\right]\\
    &\leq |f|_{Lip}\left[\alpha\int_0^\infty e^{-\alpha t}M(e^{\omega t}-1)\|A\|R\;dt+\alpha\int_0^\infty e^{-\alpha t}\int_H \|y\|\wedge(2|f|_\infty)\;\mu_t(dy)\;dt\right].
\end{align}
Therefore, by the initial value theorem for the Laplace transform we have
\begin{equation}\label{eq:sup<R}
\lim\limits_{\alpha\to\infty}\sup\limits_{\|x\|\leq R}|\alpha U_\alpha f(x)-f(x)|=0. \quad R>0,
\end{equation}
Further, let $R_0>0$ such that ${\rm supp}(f)\subset B(0,R_0)$. 
Then by assertion (i) there exists $R_1$ such that $\mathop{\inf}\limits_{\|x\|\geq R_1}\mathop{\inf}\limits_{t\in [0,1]}\|T_tx\|\geq R_0+1$, and hence for $x\in H$ with $\|x\|\geq R_1$ we have
\begin{align*}
 \alpha U_\alpha f(x)&=\alpha \int_0^{\infty}e^{-\alpha t} \int_H f(T_tx+y)\; \mu_t(dy) \; dt \\
 &\leq |f|_\infty\alpha \int_0^{1}e^{-\alpha t} \int_H 1_{B(-T_tx,R_0)}(y)\; \mu_t(dy) \; dt\\
 &\leq |f|_\infty\alpha \int_0^{1}e^{-\alpha t} \mu_t(H\setminus B(0,\|T_tx\|-R_0)) \; dt \\
 &\leq |f|_\infty\alpha \int_0^{\infty}e^{-\alpha t} \mu_t\left(H\setminus B\left(0,1\right)\right) \; dt.
\end{align*}
Therefore, since $\mathop{\lim}\limits_{t\to 0}\mu_t=\delta_0$ weakly and applying once again the initial value theorem, we get
\begin{equation}\label{eq:sup>R}
    \limsup_{\alpha \to \infty} \sup_{\|x\|\geq R_1}\alpha U_\alpha f(x) =0,
\end{equation}
and by combining \eqref{eq:sup<R} with \eqref{eq:sup>R}, we get
\begin{equation*}
    \lim\limits_{\alpha\to \infty} |\alpha U_\alpha f-f|_\infty=0.
\end{equation*}

iv). Let $\mathcal{A}\subset pLip_b(H)$ be a countable family of functions with bounded support, which generates the norm topology $\tau_{\|\|}$ on $H$, and such that $\mathcal{A}$ separates all finite measures on $H$; note that such a family always exists.
Let $\tau_{\mathcal{R}_{\mathcal{A}}}$ be the Ray topology induced by the corresponding Ray cone $\mathcal{R}_{\mathcal{A}}$ associated with $\mathcal{U}_1$ (see \Cref{rem: ray_topology}), so that  $\tau_{\mathcal{R}_{\mathcal{A}}}$ is in particular a $\mathcal{U}$-topology according to \Cref{rem: Ray_is_U}.
Then, on the one hand, since by \Cref{prop: regularity_Mehler} and the way $\mathcal{R}_{\mathcal{A}}$ is constructed we get that $\mathcal{R}_{\mathcal{A}}\subset C_b(H)$, thus $\tau_{\mathcal{R}_{\mathcal{A}}}\subset \tau_{\|\|}$,
On the other hand, if $f\in \mathcal{A}$ and $\alpha>0$ then $U_\alpha f=U_1f +(1-\alpha)U_\alpha U_1f\in \mathcal{R}_{\mathcal{A}}$, and from the uniform convergence \eqref{eq:uniform_A_bounded} we deduce that the elements of $\mathcal{A}$ are $\tau_{\mathcal{R}_{\mathcal{A}}}$ continuous.
Since $\mathcal{A}$ generates the norm topology, we obtain that $\tau_{\|\|}\subset \tau_{\mathcal{R}_{\mathcal{A}}}$, and the claim is finished.
\end{proof}

We now turn to the case when $A$ is an unbounded operator, for which we shall rely on the following condition.

\medskip
\noindent{$\mathbf{(H_{A,\mu})}$.} The following are fulfilled:
\begin{enumerate}[(i)]
    \item {\it Existence of an eigenbasis}: There exists an orthonormal basis $(e_i)_{i\geq 1}$ in $H$ such that $A^\ast e_i=\lambda_i e_i,i\geq 1$ with $(\lambda_i)_{i\geq 1}\subset \mathbb{R}\setminus\{0\}$ decreasing to $-\infty$.
    \item {\it Spectral bounds:} There exist $p\geq 2$, $a>0$, and $q>p\lambda_1$ such that
    \begin{align}
        &c_1:=\sum\limits_{i\geq 1} |\lambda_i|^{1+a} \int_0^\infty e^{-qt} \int_H |\langle y,e_i\rangle|^p\wedge 1\;\mu_t(dy)\;dt<\infty, \label{eq:V_1}\\
        \intertext{and}
        &  \sum\limits_{i\geq 1}|\lambda_i|^{a} e^{p\lambda_it}<\infty, \; t>0, \quad \mbox{ as well as } \quad \sum\limits_{i\geq 1}\frac{1}{|\lambda_i|^{\frac{2a}{p-2}}}<\infty. \label{eq:V_2}
    \end{align}
    \item {\it Moment bound}: There exists $0<\gamma\leq 1$ such that
    \begin{equation}\label{eq:V_3}
     \int_0^\infty e^{-qt} |\langle y, e_i \rangle|^{\gamma}\;\mu_t(dy) <\infty, \quad i\geq 1.
    \end{equation}
\end{enumerate}

\begin{prop}\label{prop:V_H_mu_1}
Assume that $\mathbf{(H_{A,\mu})}$ is fulfilled, let $q>p\lambda_1$, and $0<\gamma_i\leq 1, i\geq 1$ be any decreasing sequence of real numbers such that
\begin{equation}\label{eq:gamma_i}
    \sum\limits_{i\geq 1} \gamma_i^{2/(2-\gamma)}<\infty \quad \mbox{ and } \quad c_2:=\sum\limits_{i\geq 1} \gamma_i \int_0^\infty e^{-qt} |\langle y, e_i \rangle|^{\gamma}\;\mu_t(dy) <\infty,
\end{equation}
which is always possible.
Further, let us define
\begin{align}
    &v(x):=\sum\limits_{i\geq 1} |\lambda_i|^{1+a}\left(|\langle x,e_i\rangle|^p\wedge 1\right)+\sum\limits_{i\geq 1} \gamma_i |\langle x,e_i \rangle|^{\gamma}, \quad V(x):=U_qv(x), \quad x\in H\label{eq:V_H_mu_1}\\
    &H_0:=\left\{x\in H : \sum\limits_{i\geq 1}|\lambda_i|^a |\langle x,e_i\rangle|^p < \infty \right\}=\left\{x\in H : \sum\limits_{i\geq 1}|\lambda_i|^a \left(|\langle x,e_i\rangle|^p\wedge 1\right) < \infty \right\}.\label{eq:H_0}
\end{align}
Then the following assertions hold:
\begin{enumerate}[(i)]
    \item $[V<\infty]=H_0$
    \item $V$ is $\mathcal{U}_q$-excessive, $T_t(H)\subset H_0$, and $H\setminus H_0$ is polar.
    \item $\lim\limits_{\|x\|\to \infty}V(x)=\infty$, hence $V$ is a norm-like $\tau_w$-compact $q$-Lyapunov function.
    \item For all $f\in Lip(H)$ and $R>0$ we have
    \begin{equation}
        \lim\limits_{\alpha\to \infty} \sup_{x\in[V\leq R]}|\alpha U_\alpha f(x)-f(x)|=0.
    \end{equation}
    \item Let $\tau_{\|\|,V}$ be the coarser topology on $H$ which contains $\tau_{\|\|}$ and makes $U_q(v\wedge k)$ continuous for all $k\geq 1$, where recall that $v$ is given by \eqref{eq:V_H_mu_1}. 
    Then $\tau_{\|\|,V}$ is a quasi-$\mathcal{U}$-topology on $H$.
\end{enumerate}
\end{prop}
\begin{proof}
First of all, note that since $v\geq 0$, the function $V=U_qv$ is well defined with values in $[0,\infty]$. 

\medskip
(i). Let us first prove the inclusion $H_0\subset [V<\infty]$.
To this end, we proceed as follows:
\begin{align}
    V(x)&=\int_0^\infty e^{-qt} \int_H v(T_tx+y)\;\mu_t(dy)\;dt\\
    &= \int_0^\infty e^{-qt} \int_H \left(\sum\limits_{i\geq 1} |\lambda_i|^{1+a}\left(|\langle T_tx+y,e_i\rangle|^p\wedge 1\right)+\sum\limits_{i\geq 1} \gamma_i |\langle T_tx+y,e_i\rangle|^{\gamma}\right)\;\mu_t(dy)\;dt\\
    &\leq 2^{p-1}\sum\limits_{i\geq 1} |\lambda_i|^{1+a} \int_0^\infty e^{-qt} \int_H \left(|\langle x,T_t^\ast e_i\rangle|^p\wedge 1+|\langle y,e_i\rangle|^p\wedge 1\right)\;\mu_t(dy)\;dt\\
    &\quad + \int_0^\infty e^{-qt} \int_H \sum\limits_{i \geq 1} \gamma_i \left(|\langle x, T_t^\ast e_i \rangle|^{\gamma}+|\langle y, e_i \rangle|^{\gamma}\right)\;\mu_t(dy)\;dt\\
    &\leq 2^{p-1}\left[\sum\limits_{i\geq 1} |\lambda_i|^{1+a} \int_0^\infty e^{-qt}|\langle x,T_t^\ast e_i\rangle|^p\wedge 1\;dt+\sum\limits_{i\geq 1} |\lambda_i|^{1+a} \int_0^\infty e^{-qt} \int_H |\langle y,e_i\rangle|^p\wedge 1\;\mu_t(dy)\;dt\right]\\
    &\quad + \sum\limits_{i\geq 1} \gamma_i\left(\int_0^\infty e^{-qt}  \left(|\langle x, T_t^\ast e_i \rangle|^{\gamma}+ \int_H |\langle y, e_i \rangle|^{\gamma}\;\mu_t(dy)\right)\;dt\right)\\
    \intertext{and by \eqref{eq:V_1}, \eqref{eq:V_3}, and \eqref{eq:gamma_i} we can continue with}
    &\leq 2^{p-1}\left[\sum\limits_{i\geq 1} |\lambda_i|^{1+a} \int_0^\infty e^{-qt}e^{p\lambda_it}\;dt|\langle x, e_i\rangle|^p+c_1\right]+\sum\limits_{i\geq 1} \gamma_i\int_0^\infty e^{-qt} e^{\gamma\lambda_it}\;dt |\langle x, e_i \rangle|^{\gamma}+c_2\\
    &\leq 2^{p-1}\left[\sum\limits_{i\geq 1} |\lambda_i|^{1+a} \int_0^\infty e^{-qt}e^{p\lambda_it}\;dt|\langle x, e_i\rangle|^p+c_1\right]+\frac{1}{q-\gamma \lambda_1}\sum\limits_{i\geq 1} \gamma_i|\langle x, e_i \rangle|^{\gamma}+c_2\\
    &\leq 2^{p-1}\sum\limits_{i\geq 1} |\lambda_i|^{1+a} \frac{1}{q-p\lambda_i}|\langle x, e_i\rangle|^p+\frac{1}{q-\gamma \lambda_1}\left(\sum\limits_{i\geq 1} \gamma_i^{2/(2-\gamma)}\right)^{(2-\gamma)/2}\|x\|^{\gamma/2} +c_2 \label{eq:|x|gamma}\\
    &\leq 2^{p-1}\sum\limits_{i\geq 1} |\lambda_i|^{a} \frac{|\lambda_i|}{q-p\lambda_i}|\langle x, e_i\rangle|^p+\widetilde{c}(x)\\
    \intertext{and since $\sup_{i\geq 1}\frac{|\lambda_i|}{q-p\lambda_i}<\infty$ and \eqref{eq:gamma_i} is in force, we can continue as}
    &\leq \widetilde{c}\sum\limits_{i\geq 1} |\lambda_i|^{a}|\langle x, e_i\rangle|^p+\widetilde{c}(x), \quad \mbox{ where } \widetilde{c}(x)<\infty \mbox{ for all } x\in H.
\end{align}

It remains to show that $[V<\infty]\subset H_0$. To this end, using the inequality $|a-b|\wedge 1 \geq |a|\wedge1-|b|\wedge 1$, we have that
\begin{equation}
    |a+b|^p\wedge 1\geq |2^{1-p}|a|^p-|b|^p|\wedge 1\geq \left(2^{1-p}|a|^p\right)\wedge 1-|b|^p\wedge 1\geq 2^{1-p}\left(|a|^p\wedge 1\right)-|b|^p\wedge 1.
\end{equation}
\begin{align}
    V(x)&\geq \sum\limits_{i\geq 1} |\lambda_i|^{1+a} \int_0^\infty e^{-qt} \int_H |\langle T_tx+y,e_i\rangle|^p\wedge 1\;\mu_t(dy)\;dt\\
    &\geq 2^{1-p}\sum\limits_{i\geq 1} |\lambda_i|^{1+a} \int_0^\infty e^{-qt} \left(|\langle x,T_t^\ast e_i\rangle|^p\wedge 1\right)\;dt-\sum\limits_{i\geq 1} |\lambda_i|^{1+a} \int_0^\infty e^{-qt} \int_H |\langle y,e_i\rangle|^p\wedge 1\;\mu_t(dy)\;dt,\\
    \intertext{and by \eqref{eq:V_1} we can continue with}
    V(x)&\geq 2^{1-p}\sum\limits_{i\geq 1} |\lambda_i|^{1+a} \int_0^\infty e^{-qt} \left(\left(e^{pt\lambda_i}|\langle x, e_i\rangle|^p\right)\wedge 1\right)\;dt-c_1\\
    &\geq 2^{1-p}\sum\limits_{i\geq 1} |\lambda_i|^{1+a}\left(|\langle x, e_i\rangle|^p\wedge 1\right) \int_0^\infty e^{-qt} \left(e^{pt\lambda_i}\wedge 1\right)\;dt-c_1\\
    &\geq 2^{1-p}\sum\limits_{i\geq 1} |\lambda_i|^{1+a}\left(|\langle x, e_i\rangle|^p\wedge 1\right) \int_0^\infty e^{-qt} e^{-pt|\lambda_i|}\;dt-c_1\\
    &=2^{1-p}\sum\limits_{i\geq 1} |\lambda_i|^a\left(|\langle x, e_i\rangle|^p\wedge 1\right) \frac{|\lambda_i|}{q+p|\lambda_i|}-c_1\\
    &\geq \widetilde{c}\sum\limits_{i\geq 1} |\lambda_i|^a\left(|\langle x, e_i\rangle|^p\wedge 1\right) -\widetilde{c} \quad \mbox{ for some } \widetilde{c}<\infty. \label{eq:lower_bound}
\end{align}

(ii). $V$ is $\mathcal{U}_q$-excessive by \Cref{rem:U_exc}. 
The fact that $H\setminus H_0=[V=\infty]$ is polar follows from \Cref{rem:polar_characterization} if we show that $U_11_{H\setminus H_0}=0$, which in turn follows if $P_t1_{H\setminus H_0}=0$ for $t>0$ a.e.
To check the latter equality, notice first that by \eqref{eq:V_1} we have
\begin{equation}
\int_0^\infty e^{-qt} \int_H \sum\limits_{i\geq 1} |\lambda_i|^{a}\left(|\langle y,e_i\rangle|^p\wedge 1\right)\;\mu_t(dy)\;dt<\infty,
\end{equation}
therefore, $\mu_t(H\setminus H_0)=0$ for $t>0$ a.e.
Furthermore, we have that $T_t(H)\subset H_0$ for all $t>0$, since for $t>0$, $x\in H$ and $\mathcal{I}:=\{i\geq 1 : |\langle x,e_i\rangle|\vee \left(|\lambda_i|^{a} e^{pt\lambda_i}\right) \geq 1\}$ we have $|\mathcal{I}|<\infty$ and
\begin{align}
   \sum\limits_{i\geq 1} |\lambda_i|^{a}|\langle T_tx,e_i\rangle|^p&= \sum\limits_{i\geq 1} |\lambda_i|^{a}|\langle x,e_i\rangle|^p e^{pt\lambda_i}=\sum\limits_{i\in \mathcal{I}} |\lambda_i|^{a}|\langle x,e_i\rangle|^p e^{pt\lambda_i}+\sum\limits_{i\notin \mathcal{I}} |\lambda_i|^{a}|\langle x,e_i\rangle|^p e^{pt\lambda_i}\\
   &\leq \sum\limits_{i\in \mathcal{I}} |\lambda_i|^{a}|\langle x,e_i\rangle|^p e^{pt\lambda_i}+\sum\limits_{i\notin \mathcal{I}} |\langle x,e_i\rangle|^2<\infty.
\end{align}
Consequently, for $t>0$ a.e. we have
\begin{align}
P_t1_{H\setminus H_0}(x)=\int_H 1_{H\setminus H_0}(T_tx+y)\;\mu_t(dy)=\int_{H_0} 1_{H\setminus H_0}(T_tx+y)\;\mu_t(dy)=0, \quad x\in H, 
\end{align}
hence $H\setminus H_0$ is polar.

\medskip
(iii). Recall that 
\begin{align}
    V(x)&=\sum\limits_{i\geq 1} |\lambda_i|^{1+a} \int_0^\infty e^{-qt} \int_H |\langle T_tx+y,e_i\rangle|^p\wedge 1\;\mu_t(dy)\;dt+\sum\limits_{i\geq 1} \gamma_i\int_0^\infty e^{-qt} \int_H |\langle T_tx+y,e_i\rangle|^{\gamma}\;\mu_t(dy)\;dt\\
  \intertext{and by \eqref{eq:lower_bound} we also have}  
    &\geq \widetilde{c}\sum\limits_{i\geq 1} |\lambda_i|^a\left(|\langle x, e_i\rangle|^p\wedge 1\right) -\widetilde{c} +\sum\limits_{i\geq 1} \gamma_i\int_0^\infty e^{-qt} \int_H |\langle T_tx+y,e_i\rangle|^{\gamma}\;\mu_t(dy)\;dt\\
    &\geq \widetilde{c}\sum\limits_{i\geq 1} |\lambda_i|^a\left(|\langle x, e_i\rangle|^p\wedge 1\right) -\widetilde{c}+\sum\limits_{i\geq 1} \gamma_i\int_0^\infty e^{-qt} |\langle T_tx,e_i\rangle|^{\gamma}\;dt-c_2\\
    &= \widetilde{c}\sum\limits_{i\geq 1} |\lambda_i|^a\left(|\langle x, e_i\rangle|^p\wedge 1\right) -\widetilde{c}+\sum\limits_{i\geq 1} \gamma_i|\langle x, e_i\rangle|^\gamma(q-\gamma\lambda_i)^{-1}-c_2.\label{eq:lower_bound_2}
\end{align}
Note that $\mathop{\inf}\limits_{i\geq 1}(q-\gamma\lambda_i)=q-\gamma\lambda_1> 0$ by assumption. 
Further, by contradiction, assume there exists $(x_n)_n\subset H$ such that $ \|x_n\|\mathop{\rightarrow}\limits_{n}\infty$ and $\mathop{\sup}\limits_n V(x_n)<\infty$.
Then, by \eqref{eq:lower_bound_2} and since $\lim\limits_{i\to\infty}|\lambda_i|=\infty$, there must be an $L\geq 1$ such that $|\langle x_{n},e_i\rangle|< 1$ for all $i\geq L$ and $n\geq 1$.
Then
\begin{align}
\|x_{n}\|^2 &= \sum\limits_{i <L} |\langle x_{n}, e_i\rangle|^2+ \sum\limits_{i\geq L} |\lambda_i|^{-2a /p}|\lambda_i|^{2a /p}|\langle x_{n}, e_i\rangle|^2\wedge 1, \\
\intertext{hence by H\"older inequality for $p\geq 2, (p/2)^{-1}+(p/(p-2))^{-1}=1$,  we get}
\|x_{n}\|^2 &\leq \sum\limits_{i<L} |\langle x_{n}, e_i\rangle|^2+ \left(\sum\limits_{i\geq 1}\frac{1}{|\lambda_i|^{\frac{2a}{p-2}}}\right)^{(p-2)/p} \left(\sum\limits_{i\geq L} |\lambda_i|^a|\langle x_{n}, e_i\rangle|^p\wedge 1\right)^{2/p}\\
&\leq \sum\limits_{i<L}|\langle x_n,e_i\rangle|^{2-\gamma}\frac{1}{\gamma_L}(q-\gamma \lambda_1) \sum\limits_{i\geq 1} \gamma_i|\langle x_n, e_i\rangle|^\gamma(q-\gamma \lambda_i)^{-1}\\
&\quad +\left(\sum\limits_{i\geq 1}\frac{1}{|\lambda_i|^{\frac{2a}{p-2}}}\right)^{(p-2)/p} \left(\sum\limits_{i\geq 1} |\lambda_i|^a|\langle x_{n}, e_i\rangle|^p\wedge 1\right)^{2/p}.\\
\intertext{Thus by \eqref{eq:lower_bound_2}}
&\|x_{n}\|^2 \in \mathcal{O}\left(\sum\limits_{i<L}|\langle x_n,e_i\rangle|^{2-\gamma} V(x_n)+V(x_n)^{1/p}\right)=\mathcal{O}\left(\sum\limits_{i<L}|\langle x_n,e_i\rangle|^{2-\gamma}\right),
\end{align}
which is a contradiction, and hence the assertion is proved.

\medskip
iv). If $f\in {\sf Lip}_b(H)$, then
\begin{align*}
    \left|\alpha U_\alpha f(x)-f(x)\right|
    &\leq |f|_{\sf Lip}\alpha\int_0^\infty e^{-\alpha t} \|T_tx+y-x\|\wedge (2|f|_{\infty}) \mu_t(dt)\\
    &\leq |f|_{\sf Lip}\alpha\int_0^\infty e^{-\alpha t} \|T_tx-x\| dt+ |f|_{\sf Lip} \alpha\int_0^\infty e^{-\alpha t}\int_H \|y\|\wedge(2|f|_\infty) \mu_t(dy) dt.
\end{align*}
By the weak continuity of $(\mu_t)_{t\geq 0}$ and the initial value theorem, the second term in the right-hand side of the above inequality tends to zero independently of $x\in H$.
For the first term, letting $\alpha> 3\lambda_1$, we proceed as follows:
\begin{align*}
    \alpha\int_0^\infty e^{-\alpha t} \|T_tx-x\| dt 
    &\leq \left(\alpha\int_0^\infty e^{-\alpha t} \|T_tx-x\|^2 dt\right)^{1/2}=\left(\sum_{i\geq 1} \langle x, e_i\rangle^2 \alpha\int_0^\infty e^{-\alpha t} \left(e^{\lambda_i t}-1\right)^2 dt\right)^{1/2}\\
    &=\left(\sum_{i\geq 1} \langle x, e_i\rangle^2 \frac{2\lambda_i^2}{(\alpha-2\lambda_i)(\alpha-\lambda_i)}\right)^{1/2}\\
    \intertext{so that by H\"older inequality we get}
    \alpha\int_0^\infty e^{-\alpha t} \|T_tx-x\| dt 
    &\leq \left(\sum_{i\geq 1} |\langle x, e_i\rangle|^p |\lambda_i|^a\right)^{1/2p} \left(\sum_{i\geq 1} \frac{\left[2|\lambda_i|^{\left(2-\frac{2a}{p}\right)}\right]^\frac{p}{p-2}}{\left[(\alpha-2\lambda_i)(\alpha-\lambda_i)\right]^{\frac{p}{p-2}}}\right)^{\frac{p-2}{2p}}\\
    \intertext{Now, fixing $R>0$, by arguing as immediately after \eqref{eq:lower_bound_2}, there exists $L\geq 1$ such that $|\langle x,e_i\rangle|\leq 1$ for all $i\geq L$ and $x\in [V\leq R]$. Also, by (iii) we have that $\sup_{x\in [V\leq R]}\|x\|<\infty$, and using \eqref{eq:lower_bound} one can easily deduce that there exists a constant $C<\infty$ such that the first factor in the right-hand side of the above inequality is bounded by $C$. Thus,}
    \sup_{x\in [V\leq R]}\left(\alpha\int_0^\infty e^{-\alpha t} \|T_tx-x\| dt \right)
    &\leq C \left(\sum_{i\geq 1} \frac{\left[2|\lambda_i|^{\left(2-\frac{2a}{p}\right)}\right]^\frac{p}{p-2}}{\left[(\alpha-2\lambda_i)(\alpha-\lambda_i)\right]^{\frac{p}{p-2}}}\right)^{\frac{p-2}{2p}}\\
    &\leq C \left(\sum_{i\geq 1} \frac{\left[2|\lambda_i|^{\left(2-\frac{2a}{p}\right)}\right]^\frac{p}{p-2}}{(\alpha^2+2\lambda_i^2)^{\frac{p}{p-2}}}\right)^{\frac{p-2}{2p}}.\\
    \intertext{Consequently, iv) follows by letting $\alpha\to \infty$ and employing dominated convergence, since}
    \sup_{\alpha>0} \frac{\left[2|\lambda_i|^{\left(2-\frac{2a}{p}\right)}\right]^\frac{p}{p-2}}{(\alpha^2+2\lambda_i^2)^{\frac{p}{p-2}}}
    &\leq  \frac{1}{|\lambda_i|^{\frac{2a}{p-2}}}, i\geq 1, \quad \mbox{ and } \quad \sum_{i\geq 1} \frac{1}{|\lambda_i|^{\frac{2a}{p-2}}} <\infty \mbox{ by } \eqref{eq:V_2}.
\end{align*}

v). Let $0\leq f_k\leq 1,$ $k\geq 1$,  be Lipschitz funtions on $H$ such that they generate $\tau_{\|\|}$ and separate all finite measures on $H$.
Further, consider $\mathcal{A}:=\{\alpha U_\alpha f_k :  k\geq 1, q<\alpha \in \mathbb{Q}_+\}\cup\{1+U_q(v\wedge k), k\geq 1\}$, and note that by (ii) we have that the topology $\tilde\tau$ induced by $\mathcal{A}$ is a $\mathcal{U}$-topology on $H$; take $1+V=U_q(1/q+v)$ with $V$ as in the statement, instead of $V$ in \Cref{def:quasi-U}.
Furthermore, according to \Cref{prop:nest} we have that $F_n:=[V\leq n],$ $n\geq 1$,  forms a $\mathcal{U}$-nest.
Also, since by monotone convergence we have $F_n=\cap_{k\geq 1}[U_q(v\wedge k)\leq n]$, it follows that $F_n$ is $\tau_{\|\|,V}$-closed and $\tilde\tau$-closed as well, $n\geq 1$.
To conclude, it remains to show that when restricted to each $F_n$, the two topologies $\tau_{\|\|,V}$ and $\tilde\tau$ coincide. 
To this end, first notice that by Feller property \Cref{prop: regularity_Mehler}, (i), we have that $\tilde\tau \subset \tau_{\|\|,V}$.
Then, by uniform convergence guaranteed in (iv), we have that $\tau_{\|\|}|_{F_n} \subset \tilde\tau|_{F_n}$, and hence $\tau_{\|\|, V}|_{F_n} \subset \tilde\tau|_{F_n}$, $n\geq 1$. 
Consequently, the two topologies coincide on $F_n,n\geq 1$.
\end{proof}

\begin{prop}\label{prop:tau_w:quasi-U}
   If $A$ is bounded or  $\mathbf{(H_{A,\mu})}$ is fulfilled, then the weak topology $\tau_w$ is a quasi-$\mathcal{U}$-topology on $H$. 
\end{prop}

\begin{proof}
Let $\mathcal{A}\subset p\mathcal{B}(H)$ be any countable family of weakly continuous functions on $H$ whose absolute values are less than $1$, and which separates all finite measures on $H$.
Further, let $\tilde{\tau}$ be the $\mathcal{U}$-topology on $H$ induced by $U_1(\mathcal{A})$ as in \Cref{def:quasi-U}, taking e.g. $V=1$.
By \Cref{prop: regularity_Mehler}, (ii), we have that $U_1(\mathcal{A})$ consists of sequentially weakly continuous functions on $H$, and since $U_1(\mathcal{A})$ separates the points of $H$, whilst the closed balls are $\tau_w$-compact, by Stone-Weierstress theorem we deduce that $\tilde{\tau}$ coincides with $\tau_w$ on every closed ball of radius $n$ in $H$, denoted further by $\overline{B(0,n)}, n \geq 1$. 
Also, by \Cref{prop:tau_||},(ii) if $A$ is bounded, or by \Cref{prop:V_H_mu_1}, (iii) if $\mathbf{(H_{A,\mu})}$ is fulfilled, and using \Cref{prop:nest}, we deduce that $\overline{B(0,n)},n\geq1$,  form a $\mathcal{U}$-nest of $\tau_w$-compact (and hence $\tilde{\tau}$ closed) sets. Therefore, we conclude that $\tau_w$ is a quasi-$\mathcal{U}$-topology on $H$.    
\end{proof}

We have now all the ingredients to present the main result concerning generalized Mehler semigroups.

\begin{thm}\label{thm:Mehler}
Let $(P_t)_{t\geq 0}$ be the generalized Mehler semigroup constructed by \eqref{eq: Mehler}, namely
\begin{equation*}
    P_tf(x) = \int_H f(T_t x +y) \; \mu_t(dy), \quad x\in H, f\in b\mathcal{B}(H),
\end{equation*} 
and assume that either the generator ${\sf A}$ of $(T_t)_{t\geq 0}$ is bounded or condition $\mathbf{(H_{A,\mu})}$ is satisfied.
Then there exists a Hunt Markov process $X$ on $H$ with respect to the norm topology $\tau_{\|\|}$, having transition function $(P_t)_{t\geq 0}$.
Moreover, the following regularity properties hold:
\begin{enumerate}[(i)]
    \item If the generator ${\sf A}$ of $(T_t)_{t\geq 0}$ is bounded, then the Hunt process $X$ is saturated in the sense of \Cref{defi:saturated}.
    \item If condition $\mathbf{(H_{A,\mu})}$ is satisfied and $H_0\subset H$ is given by \eqref{eq:H_0}, then
    \begin{equation}\label{eq:smoothing}
        \mathbb{P}^x\left(\left\{X(t)\in H_0 \mbox{ for all } t>0\right\}\right)=1 \quad \mbox{ for all } x\in H.
    \end{equation}
\end{enumerate}
\end{thm}

\begin{proof}
First of all, by \Cref{prop: regularity_Mehler} it follows that condition $(\mathbf{H_\mathcal{C}})$ is satisfied for $\mathcal{C}=C_b(H)$.
Then, the idea is to apply \Cref{thm:main_construction} with $E=H$ and $\tau=\tau_w$.
To this end, let $V$ be given either by \eqref{eq:V_0} if $A$ is bounded, or by \eqref{eq:V_H_mu_1} if $\mathbf{(H_{A,\mu})}$ is fulfilled. 
Let us check that condition $\mathbf{(H_{V,\tau_w})}$ required in \Cref{thm:main_construction} is satisfied. First of all, by \Cref{prop:tau_||}, (ii) if $A$ is bounded, and by \Cref{prop:V_H_mu_1}, (iii) if $\mathbf{(H_{A,\mu})}$ is fulfilled, we have that $V$ is a $\tau_w$-compact $q\vee 1$-Lyapunov function.
Now, let us check the the rest of the conditions $(\mathbf{H_{V,\tau_w}})$(i)-(iii): Let $(f_k)_{k\geq 1} \subset \mathcal{B}_b$ be any family of weakly continuous functions on $H$ which separates all finite measures on $H$.
Since $V$ is bounded, $(\mathbf{H_{V,\tau_w}})$(i) and (ii) are trivial.
Further, let $\beta > q\vee 1$, and recall that by \Cref{prop: regularity_Mehler} we have that $U_\beta f_k$ is weakly sequentially continuous on $H$, hence $\tau_w$-continuous on each $\tau_w$-compact set $F_n:=\overline{[V\geq 1/n]}^{\tau_w}$ if $A$ is bounded, and on each $F_n:=\overline{[V\leq n]}^{\tau_w}$ if $\mathbf{(H_{A,\mu})}$ is fulfilled, for all $k,n\geq 1$; thus, (iii) also holds. 
Consequently, we can apply \Cref{thm:main_construction} to ensure the existence of a Hunt Markov process $X$ on $H$ with respect to the weak topology, with transition function $(P_t)_{t\geq 0}$; moreover, if $A$ is bounded, then $X$ is saturated since in this case, by \Cref{prop:tau_||}, (ii), we have that $V$ is vanishing at infinity (in the sense of \Cref{defi 3.1}, (ii)). 

The fact that $X$ is in fact a Hunt process with respect to the norm topology follows by \Cref{thm:main_construction}, since $\tau_w$ is a quasi-$\mathcal{U}$-topology according to \Cref{prop:tau_w:quasi-U}, whilst $\tau_{\|\|}$ is (coarser than) a quasi-$\mathcal{U}$-topology, according to \Cref{prop:tau_||}, (iv), and \Cref{prop:V_H_mu_1}, (v), respectively.

Finally, by \Cref{prop:V_H_mu_1}, (ii) we have that $H\setminus H_0$ is polar, so relation \eqref{eq:smoothing} follows by the probabilistic characterization of polarity expressed in \Cref{rem:polar_characterization}, (ii). 
\end{proof}

\subsection{Application to stochastic heat equations with cylindrical L\'evy noise}\label{Ss:application}

In this subsection we apply the previously obtained results, more precisely \Cref{thm:Mehler}, in order to study the following SPDE on $H:= L^2(D)$, where $D$ is a bounded domain in $\mathbb{R}^d$:
\begin{equation}\label{SDE}
    dX(t)=\Delta_0 X(t) dt + dY(t), t>0, \quad X(0)=x\in H,
\end{equation}
where $\Delta_0$ is the Dirichlet Laplacian on $D$, whilst $Y$ is a L\'evy Noise on $H$ whose characteristic exponent $\lambda:H\rightarrow \mathbb{C}$ is given by
\begin{equation}\label{eq:lambda_example}
    \lambda(\xi):=\|\Sigma_1 \xi\|_{L^2(D)}^2+ \|\Sigma_2 \xi\|_{L^2(D)}^\alpha, \quad \xi \in H.
\end{equation}
Here, $\Sigma_1$ and $\Sigma_2$ are positive definite bounded linear operators on $H$, whilst $\alpha\in (0,2)$. 

The aim is to construct an $H$-valued Hunt Markov solution to problem \eqref{SDE}, in the following sense: First we construct the generalized Mehler semigroup associated to \eqref{SDE}, for which we then show that it is the transition function of a Hunt Markov process on $H$.
To this end, we follow the strategy adopted in \cite[Section 8]{LeRo04}. 
For comparison, we emphasize that in contrast to the just mentioned work by which the process can be constructed merely on an extended Hilbert space (see \cite[Corollary 8.2]{LeRo04}), the process to be constructed here lives on the original state space $H$. 

Let $A=\Delta_0$ and $(T_t)_{t\geq 0}$ be the strongly continuous semigroup on $H=L^2(D)$ generated by $A$. 
Furthermore, let $(\lambda_k,e_k)_{k\geq 1}$ be an eigenbasis of $H$, so that $Ae_k=\lambda_ke_k, k\geq 1$, and $0>\lambda_1\geq \lambda_2\geq \dots$ are the eigenvalues of $A$. 
Recall that by Weyl's asympthotics there exists a constant $C\in (0,1)$ such that
\begin{equation}
-1/C k^{2/d} \leq \lambda_k\leq -C k^{2/d}, \quad k\geq 1.     
\end{equation}
Furthermore, we set 
\begin{equation*}
    \sigma_{i,k}:=\|\Sigma_i e_k\|, \quad 1\leq i\leq 2, k\geq 1, 
\end{equation*}
and we assume the following growth
\begin{equation}
    0\leq \sigma_{1,k}\lesssim k^{\gamma_1} \quad \mbox{and} \quad  0\leq \sigma_{2,k}\lesssim k^{\gamma_2}, \quad k\geq 1,
\end{equation}
for some $\gamma_i\in \mathbb{R}, 1\leq i\leq 2$, on which further restrictions shall be imposed.
More precisely, we introduce the following condition:

\medskip
\noindent
$\mathbf{(H_\Sigma)}$. There exist $a>0, p>2$ such that
\begin{equation}
    \frac{2(1+a)-p}{d}+p\gamma_1<-1,\quad \frac{2a}{d}+\alpha\gamma_2<-1, \quad \mbox{ and } \quad \frac{4a}{d(p-2)}>1. 
\end{equation}

\begin{coro}\label{coro:application}
The following assertions hold:  
\begin{enumerate}[(i)]
    \item Assume that 
    $\gamma_i<1/d-1/2, \; 1\leq i\leq 2$.
    Then there exits a family of (Borel) probability measures $(\mu_t)_{t\geq 0}$ on $H$ such that $\lim\limits_{t\to 0}\mu_t=\mu_0=\delta_0$, and
    \begin{equation}
        \hat{\mu}_t(\xi):=\int_H e^{i\langle \xi, y\rangle} \mu_t(dy)=e^{\int_0^t \lambda( T_s^\ast \xi) ds}, \quad \xi\in H, t\geq 0,
    \end{equation}
    where $\lambda$ is given by \eqref{eq:lambda_example}.
    \item If $\mathbf{(H_\Sigma)}$ is fulfilled then there exists a Hunt process $X$ on $H$ whose transition function is the generalized Mehler semigroup induced by $(T_t)_{t\geq 0}$ and $(\mu_t)_{t\geq 0}$, as defined by \eqref{eq: Mehler}. 
    Moreover,
    \begin{equation}
        \mathbb{P}^x\left(\left\{X(t)\in H_0 \mbox{ for all } t>0\right\}\right)=1 \quad \mbox{ for all } x\in H,
    \end{equation}
    where $H_0\subset H$ is given by \eqref{eq:H_0} with $a$ and $p$ given in $\mathbf{(H_\Sigma)}$, and all the properties (i)-(v) from \Cref{prop:V_H_mu_1} hold.
    \item If $\gamma_1<1/d-1/2$ and $\gamma_2<-1/\alpha$, then there exists a Hunt process $X$ on $H$ whose transition function is the generalized Mehler semigroup induced by $(T_t)_{t\geq 0}$ and $(\mu_t)_{t\geq 0}$ given by \eqref{eq: Mehler}.
\end{enumerate}
\end{coro}

\begin{proof}
 (i). We follow \cite[Section 8]{LeRo04}, in particular it is enough to show that  for each $t> 0$, 
 \begin{equation*}
 N_1(\xi):=\int_0^t \|\Sigma_1T_s^\ast \xi\|^2 ds \quad \mbox{ and } \quad N_2(\xi):=\int_0^t \|\Sigma_2T_s^\ast \xi\|^\alpha ds, \quad \xi\in H,
 \end{equation*}
 are Sazonov continuous functionals on $H\equiv H^\ast$; see e.g. \cite{Saz58}, \cite{Li82}, or \cite{FuRo00} for details on the Sazonov topology. 
 To this end, it is enough to show that there exits $V_1,V_2$ Hilbert-Schmidt operators on $H$, and $\beta_1,\beta_2$ positive constants such that 
 \begin{equation}
    0\leq N_i(\xi) \lesssim \|V_i \xi\|^{\beta_i}, \quad \xi\in H, i=1,2.
 \end{equation}
Let us treat $N_1$ first:
\begin{align*}
    N_1(\xi)
    &=\int_0^t \|\sum\limits_{k\geq 1}e^{s\lambda_k}\langle\xi,e_k\rangle\sigma_{1,k} e_k\|^2 ds= \int_0^t \left(\sum\limits_{k\geq 1}e^{2s\lambda_k}\sigma_{1,k}^2\langle\xi,e_k\rangle^2\right) ds\\
    &\leq \sum\limits_{k\geq 1}\int_0^t e^{-2sCk^{2/d}} ds \;  c_1 k^{2\gamma_1}\langle\xi,e_k\rangle^2\leq \sum\limits_{k\geq 1}c_1 \frac{1}{2Ck^{2/d}} k^{2\gamma_1}\langle\xi,e_k\rangle^2\\
    &=\frac{c_1}{2C}\sum\limits_{k\geq 1} k^{2\gamma_1-2/d}\langle\xi,e_k\rangle^2=\frac{c_1}{2C}\|V_1\xi\|^2,
\end{align*}
where $V_1 e_k:=k^{\gamma_1-1/d}, k\geq 1$ defines a Hilbert-Schmidt operator on $H$, since $2\gamma_1-2/d<-1$.
Thus, $N_1$ is Sazonov continuous.

To treat $N_2$, since $\alpha<2$, we first use H\"older inequality and then we rely on the previous computations, as follows:
\begin{align*}
    N_2(\xi)
    &\leq t^{1-\alpha/2} \left(\int_0^t \|\Sigma_2T_s^\ast \xi\|^2 ds\right)^{\alpha/2}=t^{1-\alpha/2}\left(\frac{c_2}{2C}\right)^{\alpha/2}\|V_2\xi\|^\alpha,
\end{align*}
where $V_2 e_k:=k^{\gamma_2-1/d}, k\geq 1$ defines a Hilbert-Schmidt operator on $H$, since $2\gamma_2-2/d<-1$. 
Consequently, the first statement is proved.

\medskip
(ii). First of all, if $\mathbf{(H_\Sigma)}$ holds then by a straightforward computation it follows that $\gamma_1<1/d-1/2, 1\leq i\leq 2$, so the conclusion in (i) is ensured, and hence the corresponding generalized Mehler semigroup cand be constructed on $H$.
Then, the idea is to apply \Cref{thm:Mehler} by checking condition $\mathbf{(H_{A,\mu})}$.
Since $\mathbf{(H_{A,\mu})}$, (i), is clearly satisfied, let us check (ii). 
Regarding condition \eqref{eq:V_1}, it is straightforward to see that
\begin{align*}
\int_H |\langle y,e_k\rangle|^p\wedge 1\;\mu_t(dy)= \int_{\mathbb{R}} |x|^p\wedge 1\mu_t^{(k)}(dx),
\end{align*}
where for each $t>0$, $\mu_t^{(k)}$ is a probability distribution on $\mathbb{R}$ whose characteristic exponent is given by 
\begin{equation*}
    \lambda_t^{(i)}(s)=  \left|s \sigma_{1,k}\right|^2\int_0^te^{-2|\lambda_k| s} ds+ \left|s \sigma_{1,k}\right|^\alpha\int_0^te^{-\alpha|\lambda_k| s} ds,\quad s \in \mathbb{R}
\end{equation*}
hence $\mu_t^{(k)}$ is the distribution of $\sigma_{1,k}\left(\int_0^te^{-2|\lambda_k| s} ds\right)^{1/2} Y_1+\sigma_{2,k}\left(\int_0^te^{-\alpha |\lambda_k| s} ds\right)^{1/\alpha} Y_2$, where $Y_1$ is a standard normal variable whilst $Y_2$ is a centered symmetric $\alpha$-stable random variable with scale $1$, $Y_1$ and $Y_2$ being independent.
Consequently,
\begin{align}
    &\int_H |\langle y,e_k\rangle|^p\wedge 1\;\mu_t(dy)\\    &=\mathbb{E}\left[\left|\sigma_{1,k}\left(\int_0^te^{-2|\lambda_k| s} ds\right)^{1/2} Y_1+\sigma_{2,k}\left(\int_0^te^{-\alpha|\lambda_k| s} ds\right)^{1/\alpha} Y_2\right|^p\wedge 1 \right]\\
    &\leq 2^{p-1}\mathbb{E}\left[\left|\sigma_{1,k}Y_1\left(\int_0^te^{-2|\lambda_k| s} ds\right)^{1/2}\right|^p\wedge 1 \right]+2^{p-1}\mathbb{E}\left[\left|\sigma_{2,k}Y_2 \left(\int_0^te^{-\alpha|\lambda_k| s} ds\right)^{1/\alpha}\right|^p\wedge 1 \right],\\
    \intertext{and using that $|u|^p\wedge 1 \leq |u|^\alpha\wedge 1$ because $\alpha<2\leq p$, we continue with}
    &\leq 2^{p-1}\mathbb{E}\left[\left|\sigma_{1,k}Y_1\left(\int_0^te^{-2|\lambda_k| s} ds\right)^{1/2}\right|^p\wedge 1 \right]+2^{p-1}\mathbb{E}\left[\left(\left|\sigma_{2,k}Y_2\right|^\alpha \int_0^te^{-\alpha|\lambda_k| s} ds\right)\wedge 1 \right],\\
    \intertext{so, taking $\beta<\alpha$ we can continue with}
    &\leq 2^{p/2-1}\frac{\sigma_{1,k}^p}{|\lambda_k|^{p/2}}\mathbb{E}\left[\left|Y_1\right|^p \right]+\frac{2^{p-1}}{\alpha}\frac{\left|\sigma_{2,k}\right|^\beta}{|\lambda_k|}\mathbb{E}\left[\left|Y_2\right|^\beta\right].
\end{align}
Note that since $\beta<\alpha$ we have that $\mathbb{E}\left[|Y_2|^\beta\right]<\infty$.
Therefore, $c_1$ defined in \eqref{eq:V_1} with $a$ and $p$ as in $\mathbf{(H_\Sigma)}$ satisfies
\begin{align}
    c_1
    &\lesssim \sum\limits_{k\geq 1} |\lambda_k|^{1+a-p/2}\sigma_{1,k}^p+\sum\limits_{k\geq 1} |\lambda_k|^{a}\sigma_{2,k}^\beta \leq \sum\limits_{k\geq 1} \left(k^{\frac{2(1+a)-p}{d}+p\gamma_1}+k^{\frac{2a}{d}+\beta\gamma_2}\right)<\infty,
\end{align}
since by $\mathbf{(H_\Sigma)}$, the two powers of $k$ are strictly less than $-1$.

The rest of assertion (ii) follows from \Cref{thm:Mehler}.

\medskip
(iii) The idea is to treat the cases $\Sigma_2=0$, respectively $\Sigma_1=0$, separately, and then construct the process for the general case from the ones corresponding to the two cases just mentioned, to be obtained by the previous assertion (ii).

Let us first assume that $\Sigma_2=0$, so that $\mathbf{(H_\Sigma)}$ reduces to:
{\it There exist $a>0, p>2$ such that
\begin{equation}\label{eq:S_1}
    \frac{2(1+a)-p}{d}+p\gamma_1<-1 \quad \mbox{ and } \quad \frac{4a}{d(p-2)}>1,
\end{equation}
}
or equivalently, by writing the two inequalities in terms of $p$ only after some straightforward calculations,
\begin{equation}\label{eq:s_1}
    \gamma_1<1/d-1/2-2/(dp).
\end{equation}
But since $\gamma_1<1/d-1/2$ is satisfied by assumption, we get that \eqref{eq:s_1} and thus \eqref{eq:S_1} are satisfied for some sufficiently large $p\geq 2$ and $a>0$.
Consequently, from assertion (ii) we get: {\it There exists a Hunt process $\left(X_{(1)},\mathbb{P}_{(1)}^x,x\in H\right)$ on $H$ whose transition function is the generalized Mehler semigroup induced by $(T_t)_{t\geq 0}$ and $(\mu_t)_{t\geq 0}$, for $\lambda$ given by \eqref{eq:lambda_example} with $\Sigma_2=0$.}

If $\Sigma_1=0$, then $\mathbf{(H_\Sigma)}$ reduces to:
{\it There exist $a>0, p>2$ such that
\begin{equation}\label{eq:S_1}
    \frac{2a}{d}+\alpha\gamma_2<-1 \quad \mbox{ and } \quad \frac{4a}{d(p-2)}>1,
\end{equation}
}
which is clearly satisfied for $a>0$ and $p>2$ sufficiently small, since by assumption $\gamma_2<-1/\alpha$.
Consequently, from assertion (ii) we get: {\it There exists a Hunt process $\left(X_{(2)},\mathbb{P}_{(2)}^x,x\in H\right)$ on $H$ whose transition function is the generalized Mehler semigroup induced by $(T_t)_{t\geq 0}$ and $(\mu_t)_{t\geq 0}$, for $\lambda$ given by \eqref{eq:lambda_example} with $\Sigma_1=0$.}

Now let us consider the general case in which both $\Sigma_1$ and $\Sigma_2$ can be non-zero, and start by taking two independent copies $\tilde{X}_{(1)}, \tilde{X}_{(2)}$ of $\left(X_{(1)},\mathbb{P}_{(1)}^0\right)$ and $\left(X_{(2)},\mathbb{P}_{(2)}^0\right)$, defined on a (common) probability space $\left(\tilde{\Omega}, \tilde{\mathcal{F}}, \tilde{\mathbb{P}}\right)$.
Further, we set
\begin{equation}
    \tilde{X}^x(t):=T_t x+\tilde{X}_{(1)}(t)+\tilde{X}_{(2)}(t), \quad t\geq 0, x\in H,
\end{equation}
so that $\left(\tilde{X}^x,\tilde{\mathbb{P}}\right)$ is a process satisfying the simple Markov property, with c\`adl\`ag paths in $H$, and starting at $x\in H$.
Moreover, one can easily see that
\begin{equation*}
    \mathbb{E}^{\tilde{\mathbb{P}}}\left\{f(\tilde{X}^x(t))\right\}=P_tf(x), \quad t\geq 0,x\in H, f\in b\mathcal{B},
\end{equation*}
where $\left(P_t\right)_{t\geq 0}$ is the generalized Mehler semigroup induced by $(T_t)_{t\geq 0}$ and $(\mu_t)_{t\geq 0}$.
Consequently, we can apply \cite[Proposition 2.1]{BeCiRo24a} to deduce that there exists a right process $X$ with c\'adl\`ag paths in $H$ (with respect to the norm topology) and transition function $(P_t)_{t\geq 0}$.
The fact that $X$ is also Hunt follows by e.g. \Cref{rem:cadlag+Feller=Hunt}, (ii).
\end{proof}

\begin{rem}\label{rem:optimal}
Note that in the previous example, if we take $\Sigma_2=0$, then the imposed condition $\gamma_1<1/d-1/2$ is precisely the condition obtained right after \eqref{eq:Weyl}, which is also equivalent to \eqref{eq:condition_stoch_conv}. 
Thus, at least in this particular case, the general result \Cref{thm:Mehler} renders optimal conditions.
\end{rem}

\medskip
\paragraph{Acknowledgements.} 
M. R. gratefully acknowledges financial support by the Deutsche Forschungsgemeinschaft
(DFG, German Research Foundation) -- Project-ID 317210226 -- SFB 1283. 
L. B. was supported by a grant of the Ministry of Research, Innovation and Digitization, CNCS - UEFISCDI,
project number PN-III-P4-PCE-2021-0921, within PNCDI III. 
I. C. acknowledges support from the Ministry of Research, Innovation and Digitization (Romania), grant CF-194-PNRR-III-C9-2023.\\

\end{document}